\documentclass[a4paper,12pt,reqno]{amsart}  

\usepackage{amsmath}
\usepackage[all]{xy}
\usepackage{tikz-cd}
\usepackage{amssymb}
\usepackage{amsthm}
\usepackage{amsfonts}
\usepackage{mathtools}
\usepackage{comment}
\usepackage{mathrsfs}
\usepackage{appendix}
\usepackage[noadjust]{cite}
\usepackage{enumerate}
\usepackage{here}
\usepackage{autobreak}
\usepackage{wrapfig}
\usepackage{graphicx}
\usepackage{lscape}
\usepackage[colorlinks]{hyperref}  
\usepackage{xcolor}
\hypersetup{
	bookmarksnumbered=true,
    colorlinks=true,
    citecolor=blue,
    linkcolor=purple,
    urlcolor=orange,
}

\usepackage[capitalize, nameinlink]{cleveref}

\numberwithin{equation}{section}

\theoremstyle{plain}

\newtheorem{thm}{Theorem}[section]

\newtheorem{lem}[thm]{Lemma}
\crefname{lem}{Lemma}{Lemmas}
\newtheorem{prop}[thm]{Proposition}
\crefname{prop}{Proposition}{Propositions}
\newtheorem{cor}[thm]{Corollary}
\newtheorem{claim}{Claim}

\newtheorem{introthm}{Theorem}[section]
\crefname{introthm}{Theorem}{Theorems}

\theoremstyle{definition}

\newtheorem{setup}[thm]{Setup}
\crefname{setup}{Setup}{Setups}
\newtheorem{eg}[thm]{Example}
\newtheorem*{Ack}{Acknowledgements}

\newtheorem*{NoCon}{Notation and Conventions}
\newtheorem*{Out}{Outline of this paper}

\theoremstyle{remark}
\newtheorem{rem}[thm]{Remark}

\DeclareMathOperator{\Ext}{Ext}

\DeclareMathOperator{\Nef}{Nef}

\DeclareMathOperator{\NEbar}{\overline{NE}}
\DeclareMathOperator{\Spec}{Spec}
\DeclareMathOperator{\Bl}{Bl}

\DeclareMathOperator{\Exc}{Exc}

\newcommand\Hom{\mathop{\mathrm{Hom}}\nolimits}

\newcommand{\PP}[1]{\mathbb{P}^{#1}}

\newcommand\pf{\mathfrak p}

\DeclareMathOperator\Ass{Ass}
\DeclareMathOperator\Proj{Proj}

\newcommand\dual{\raise0.9ex\hbox{$\scriptscriptstyle\vee$}}

\newcommand\linsys[1]{\lvert#1\rvert}

\newcommand\con[2]{\CC_{{#1}/{#2}}}
\DeclareMathOperator{\codim}{codim}
\DeclareMathOperator{\mult}{mult}
\DeclareMathOperator{\Supp}{Supp}
\DeclareMathOperator{\coeff}{coeff}

\newcommand\veq{\rotatebox{-90}{$=$}}
\newcommand\vsupset{\rotatebox{-90}{$\supseteq$}}
\newcommand\vdotsb{\rotatebox{-90}{$\dotsb$}}

\usepackage{mleftright}
\usepackage{multirow}
\usepackage{braket}
\usepackage{enumitem}

\crefname{enumi}{}{}
\creflabelformat{enumi}{#2(#1)#3}

\crefname{numii}{}{}
\creflabelformat{enumii}{#2(#1)#3}

\crefformat{equation}{(#2#1#3)}
\crefmultiformat{equation}
  {(#2#1#3)}
  { and (#2#1#3)}
  {, (#2#1#3)}
  { and (#2#1#3)}

\makeatletter
\renewcommand\p@enumiii{\theenumi\theenumii.}
\makeatother

\newcommand{\CB}{\mathbb{C}}

\newcommand{\QB}{\mathbb{Q}}
\newcommand{\RB}{\mathbb{R}}
\newcommand{\PB}{\mathbb{P}}

\newcommand{\CC}{\mathcal{C}}

\newcommand{\IC}{\mathcal{I}}

\newcommand{\OC}{\mathcal{O}}

\author{Yuto Masamura}
\address[Y.Masamura]{Graduate School of Mathematical Sciences, the University of Tokyo, 3-8-1 Komaba, Meguro-ku, Tokyo 153-8914, Japan}
\email{\href{mailto:masamura@ms.u-tokyo.ac.jp}{masamura@ms.u-tokyo.ac.jp}}

\author{Tomoki Yoshida}
\address[T.Yoshida]{Department~of~Mathematics, School~of~Science~and~Engineering, Waseda~University, Ohkubo~3-4-1, Shinjuku, Tokyo~169-8555, Japan}
\email{\href{mailto:tomoki_y@asagi.waseda.jp}{tomoki\_y@asagi.waseda.jp}}

\title[Smooth varieties admitting small contractions]{A construction of smooth varieties admitting small contractions}
\date{November 11, 2026}


\begin{document}
\begin{abstract}
    Given a smooth variety together with two smooth subvarieties, we construct, via two successive blowups, smooth varieties admitting small contractions.
    This generalizes Kawamata's example of small contraction in dimension $4$.
    We also construct the flip of the contraction
    explicitly.
    As an application, from products of two del Pezzo surfaces we obtain smooth weak Fano fourfolds that admit small contractions with Picard number up to $10$.
\end{abstract}
\maketitle
\setcounter{tocdepth}{2} 
\tableofcontents
\setcounter{section}{-1}
\section{Introduction}\label{Section: introduction}

We work over $\CB$, the field of complex numbers.

\subsection{General construction of smooth varieties admitting small contractions}
\label{subsection: intro construction}

Small contractions are one of the fundamental types of birational contractions that arise as steps in the minimal model program (MMP), together with flips and divisorial contractions.
A \emph{small contraction} is a projective birational morphism whose exceptional locus has codimension at least two.
While the contraction theorem guarantees their existence under suitable hypotheses, constructing explicit examples remains nontrivial.
Moreover, since $K$-extremal small contractions of smooth varieties exist only in dimension four or higher, constructing smooth examples is of particular interest.

In dimension four, Kawamata \cite{kawamata_1989_small_contractions_of_fourdimensional_algebraic_manifolds} constructed examples of $K$-extremal small contractions on smooth projective fourfolds using successive blowups.
Beginning with a smooth projective fourfold, he first blew up a curve and then a surface (see also Debarre \cite[6.19]{debarre_2001_higher-dimensional_algebraic_geometry}).
Tsukioka \cite{tsukioka_2022_on_weak_fano_manifolds_with_small_contractions_obtained_by_blowups_of_a_product_of_projective_spaces} generalized this construction to higher dimensions: starting from a smooth projective variety of dimension $\ge4$, he blew up a curve and then a codimension-two subvariety.
He applied this to $\PP n\times\PP1$ and studied when the resulting varieties are Fano or weak Fano.
Tsukioka \cite{tsukioka_2023_some_examples_of_log_fano_structures_on_blowups_along_subvarieties_in_products_of_two_projective_spaces} also applied an analogous construction to $\PP{n_1}\times\PP{n_2}$ and proved the existence of $K$-extremal small contractions.
In that work, the centers of the blowups were subvarieties of codimension $n_1+1$ and $n_2$, each contained in a fiber of the projections.

Our first main result generalizes these constructions.
Let $X''$ be a smooth quasi-projective variety containing smooth subvarieties $A''$ and $B''$.
We blow up $X''$ along $A''$ and then the strict transform of $B''$.
Our construction offers two key improvements.
First, it allows for flexible codimensions of the centers.
In contrast, Tsukioka's results require
\[
    \codim A''=\dim X''-1,\quad\codim B''=2,
\]
or $\codim A''+\codim B''=\dim X''+1$ for $X''=\PP{n_1}\times\PP{n_2}$.
Second, the intersection $A''\cap B''$ may have positive dimension, whereas Tsukioka required it to be zero-dimensional.

\begin{introthm}[\cref{Theorem: existence of small contraction}]
    \label{introthm: existence of small contraction}
    Let $X''$ be a smooth quasi-projective variety of dimension $\ge3$, and let $A'',B''\subseteq X''$ be smooth subvarieties of codimension $a,b\ge2$.
    Assume that
    \begin{itemize}
        \item $A''\nsupseteq B''$,
        \item the scheme-theoretic intersection $Z''\coloneq A''\cap B''$ is nonempty and smooth, and
        \item each irreducible component of $Z''$ has codimension $<a+b$ in $X''$.
    \end{itemize}
    Let $X$ be the blowup of $X''$ successively along $A''$ and along the strict transform of $B''$.
    Then there exists a birational contraction $\psi\colon X\to X_0$ over $X''$ with $\rho(X/X_0)=1$.
    Moreover,
    \begin{enumerate}
        \item $\psi$ is a small contraction if and only if $A''\nsubseteq B''$, and
        \item $\psi$ is $K_X$-extremal if and only if $a>b$.
    \end{enumerate}
\end{introthm}

We also provide an explicit description of the target $X_0$.
We show that $X_0$ is isomorphic to the blowup of $X''$ along the union $A''\cup B''$.
This description remains valid even if $a<b$, which enables us to explicitly construct the flip of $\psi$.

\begin{introthm}[\cref{Theorem: existence of small contraction,corollary: flip}]
    \label{introthm: flip}
    In the setting of \cref{introthm: existence of small contraction}, the following hold:
    \begin{enumerate}
        \item The variety $X_0$ is the blowup of $X''$ along the union $A''\cup B''$.
        \item If $A''\nsubseteq B''$ and $a>b$, then the flip $X^+$ of $\psi$ is the blowup of $X''$ successively along $B''$ and along the strict transform of $A''$.
    \end{enumerate}
\end{introthm}

The proof of \cref{introthm: existence of small contraction} relies on determining the relative nef cone and the cone of curves of $X$ over $X''$.
Specifically, we prove that
\begin{align*}
\Nef(X/X'') &= \RB^+[-E,-E-F],\\
\NEbar(X/X'') &= \RB^+[e,f].
\end{align*}
Here, $E$ is the pullback to $X$ of the exceptional divisor of the first blowup $X'\to X''$, and $F$ is the exceptional divisor of the second blowup $X\to X'$.
Moreover, $e$ and $f$ are curves in $X$ satisfying
\[
    -E.e=-(E+F).f=1,\quad-E.f=-(E+F).e=0.
\]
The birational contraction $\psi$ in \cref{introthm: existence of small contraction} contracts the extremal ray $\RB^+[e]$.
Equivalently, it is the contraction defined by the semiample$/X''$ divisor $-E-F$.
Note that the other extremal ray $\RB^+[f]$ corresponds to the blowdown $X\to X'$.

\subsection{Nef divisors on two successive blowups}

We retain the notation from \cref{introthm: existence of small contraction}.
While the proof of \cref{introthm: existence of small contraction} relies solely on the relative geometry of $X$ over $X''$, our second main result investigates the global geometry of $X$, specifically the nef cone $\Nef(X)$.
We provide sufficient conditions for the divisors
\[
	\varphi^*H''-E,\quad\varphi^*H''-E-F
\]
to be nef, where $\varphi\colon X\to X''$ is the composite blowup and $H''$ is an $\RB$-divisor on $X''$.
This result generalizes the approach of Tsukioka \cite{tsukioka_2022_on_weak_fano_manifolds_with_small_contractions_obtained_by_blowups_of_a_product_of_projective_spaces,tsukioka_2023_some_examples_of_log_fano_structures_on_blowups_along_subvarieties_in_products_of_two_projective_spaces}.

\begin{introthm}[\cref{lemma: nefness of H-E,lemma: nefness of H-E-F}]
    \label{introthm: nefness of H-E-F}
    Let $X''$ be a smooth projective variety, and let $A'',B''\subseteq X''$ be smooth subvarieties of codimension $a,b\ge1$ respectively.
    Assume that
    \begin{itemize}
        \item $A''\nsupseteq B''$, and
        \item $Z''\coloneq A''\cap B''$ is irreducible, smooth, and of codimension $a+b-c$ in $X''$ for some $c\ge0$.
    \end{itemize}
    Let $\pi'\colon X'\to X''$ be the blowup along $A''$ with exceptional divisor $E'$, and let $\pi\colon X\to X'$ be the blowup along the strict transform $B'\coloneq (\pi')_*^{-1}B''$ with exceptional divisor $F$.
    Set $E\coloneq \pi^*E'$ and $\varphi\coloneq\pi'\circ\pi\colon X\to X''$.
    Let $H''$ be an $\RB$-divisor on $X''$.
    \begin{enumerate}
        \item Assume that there exists a chain of smooth subvarieties
            \[
                X''=X''_0\supseteq X''_1\supseteq\dotsb\supseteq X''_a=A''
            \]
            such that
            \begin{itemize}
                \item $X''_{i+1}$ is a prime divisor on $X''_i$ for $0\le i<a$, and
                \item $H''\rvert_{X''_i}-X''_{i+1}$ is nef on $X''_i$ for $0\le i<a$.
            \end{itemize}
            Then the divisor $\varphi^*H''-E$ on $X$ is nef.
        \item Assume that there exist chains of smooth subvarieties
            \begin{align*}
                X''=X''_0\supseteq X''_1\supseteq\dotsb\supseteq X''_c&=A''_c\supseteq A''_{c+1}\supseteq\dotsb\supseteq A''_a=A'',\\
                X''_c&=B''_c\supseteq B''_{c+1}\supseteq\dotsb\supseteq B''_b=B''
            \end{align*}
            such that, setting $Z''_{i,j}\coloneq A''_i\cap B''_j$ for $c\le i\le a$ and $c\le j\le b$,
            \begin{itemize}
                \item $X''_{i+1}$ is a prime divisor on $X''_i$ for $0\le i<c$,
                \item $Z''_{i+1,j}$ is a smooth prime divisor on $Z''_{i,j}$ for $c\le i<a$ and $c\le j\le b$, and $Z''_{i,j+1}$ is a smooth prime divisor on $Z''_{i,j}$ for $c\le i\le a$ and $c\le j<b$,
                \item $H''\rvert_{X''_i}-X''_{i+1}$ is nef on $X''_i$ for $0\le i<c$,
                \item $H''\rvert_{Z''_{i,j}}-Z''_{i+1,j}-Z''_{i,j+1}$ is nef on $Z''_{i,j}$ for $c\le i<a$ and $c\le j<b$.
            \end{itemize}
            Then the divisor $\varphi^*H''-E-F$ on $X$ is nef.
    \end{enumerate}
\end{introthm}

\subsection{Applications to smooth weak Fano fourfolds}

Fano varieties play a central role in the MMP.
The classification of smooth Fano varieties, in particular, is a long-standing problem in algebraic geometry.
While the classification of smooth Fano threefolds is complete, the higher-dimensional case remains an active area of research.

Tsukioka \cite{tsukioka_2022_on_weak_fano_manifolds_with_small_contractions_obtained_by_blowups_of_a_product_of_projective_spaces,tsukioka_2023_some_examples_of_log_fano_structures_on_blowups_along_subvarieties_in_products_of_two_projective_spaces} applied the blowup constructions described in \cref{subsection: intro construction} to produce higher-dimensional examples of smooth Fano, weak Fano, and Fano type varieties.
In \cite{tsukioka_2022_on_weak_fano_manifolds_with_small_contractions_obtained_by_blowups_of_a_product_of_projective_spaces}, he considered the blowup of $X''=\PP n\times\PP1$ along a fiber $A''\cong\PP1$ and a codimension-two subvariety $B''$ in a specified class, establishing criteria for the resulting variety $X$ to be Fano or weak Fano.
Subsequently, in \cite{tsukioka_2023_some_examples_of_log_fano_structures_on_blowups_along_subvarieties_in_products_of_two_projective_spaces} he extended this to $\PP{n_1}\times\PP{n_2}$, with different centers $A''$ and $B''$, determining conditions under which $X$ is Fano, weak Fano, or of Fano type.

Our work is further motivated by Casagrande's recent results \cite{casagrande_2024_fano_4folds_with_b2_12_are_products_of_surfaces,casagrande_2025_towards_the_classification_of_Fano_4-folds_with_b2_7} on the Picard number of smooth Fano fourfolds:
\begin{introthm}[{\cite[Theorem 1.1]{casagrande_2025_towards_the_classification_of_Fano_4-folds_with_b2_7}}]
    Let $X$ be a smooth Fano fourfold of Picard number $\rho(X)\ge10$.
    Then $X$ is isomorphic to the product of two smooth del Pezzo surfaces.
\end{introthm}
Aside from products, only five families of smooth Fano fourfolds with Picard number $7$, $8$, or $9$ are known (see \cite[p.1]{casagrande_2025_towards_the_classification_of_Fano_4-folds_with_b2_7} and \cite{casagrande_codogni_fanelli_2019_the_blowup_of_bbb_p4_at_8_points_and_its_fano_model_via_vector_bundles_on_a_del_pezzo_surface}).
Constructing new examples with large Picard number is thus of significant interest.
Note that the existence of a small contraction on a variety implies that it is not a product of del Pezzo surfaces.

In this paper, we apply the blowup construction to the product $X''=X_1\times X_2$ of two del Pezzo surfaces.
We prove that the resulting variety $X$ admits a small contraction (\cref{introthm: existence of small contraction}), which implies that $X$ is not a product of del Pezzo surfaces.
By \cref{introthm: nefness of H-E-F}, we also obtain precise criteria for $X$ to be Fano, weak Fano, or of Fano type.
Thus, our construction extends Tsukioka's work for $X''=\PP2\times\PP2$, and, in view of Casagrande's results, it provides weak Fano fourfolds and fourfolds of Fano type with larger Picard numbers that are not products of del Pezzo surfaces.

\begin{introthm}[\cref{cor: characterization of Fano}]
    \label{introthm: weak Fano 4-fold of Picard number 8 with a small contraction}
    Let $0\le r_1,r_2\le8$ be integers.
    Let $X_i$ be the blowup of $\PP2$ at $r_i$ points in general position.
    Let $a_1\in X_1$ be a general point, $A_2\subseteq X_2$ be the strict transform of a general line in $\PP2$, and let $b_2\in A_2$ be a general point.
    Set $X''\coloneq X_1\times X_2$, $A''\coloneq\{a_1\}\times A_2$, and $B''\coloneq X_1\times\{b_2\}$.
    Let $X$ be the blowup of $X''$ successively along $A''$ and along the strict transform of $B''$.

    Then $X$ is a smooth projective fourfold admitting a $K_X$-extremal small contraction.
    Moreover,
    \begin{enumerate}
        \item $X$ is Fano if and only if $r_1=r_2=0$.
        \item $X$ is weak Fano if and only if $r_1\le5$ and $r_2\le1$.
        \item $X$ is of Fano type if and only if
            \[
            \begin{cases}
                r_1\le 6,\ r_2\le 7, \quad\text{or}\\
                r_1=7,\ r_2\le 5.
            \end{cases}
            \]
    \end{enumerate}
\end{introthm}

Since each such $X$ admits a $K_X$-extremal small contraction, it is not a product of del Pezzo surfaces.
For instance, taking $(r_1,r_2)=(5,1)$ yields a smooth weak Fano fourfold with $\rho(X)=10$ that carries such a small contraction.
However, since these examples are only weak Fano and not Fano, they do not directly complement Casagrande's result.
Producing genuinely Fano examples of large Picard number in this way appears to require a different choice of the ambient variety $X''$ and of the centers, or substantially different methods, and we leave this as an open problem.


\begin{Out}
    \Cref{section: preliminaries} collects preliminaries.
    In \cref{section: nefness of divisors on blowups}, we present general results on the nefness of divisors on blowups and prove \cref{introthm: nefness of H-E-F}.
    \Cref{section: construction} is devoted to the general blowup construction of varieties admitting small contractions.
    We analyze the relative nef cone and the cone of curves to prove \cref{introthm: existence of small contraction,introthm: flip}.
    Finally, \cref{section: weak Fano 4-fold} applies this theory to products of del Pezzo surfaces to prove \cref{introthm: weak Fano 4-fold of Picard number 8 with a small contraction}.
\end{Out}

\begin{NoCon}
    Conventions and notation used in this paper are as follows:
    \begin{itemize}
        \item All subvarieties are assumed to be closed.
        \item Intersections $A\cap B$ are taken scheme-theoretically.
        \item $N^1(X/S)$ denotes the real vector space of numerical classes of divisors on $X$ relative to $S$.
        \item $N_1(X/S)$ denotes the real vector space of numerical classes of curves on $X$ that are contracted to $S$.
        \item $\RB^+[A_1, \ldots, A_m]$ denotes the convex cone generated by $A_1,\dotsc,A_m$.
        \item A \emph{$K_X$-extremal} contraction is a contraction $\pi\colon X\to Y$ such that $\rho(X/Y)=1$ and $-K_X$ is $\pi$-ample.
    \end{itemize}
\end{NoCon}

\begin{Ack}
    This work was initiated during the authors’ visit to Taipei in November--December 2024.
    The authors thank Professor Hsueh-Yung Lin for his hospitality.
    The first author thanks his advisor, Professor Keiji Oguiso, for his continuous support.
    The second author is grateful to Professor Yasunari Nagai for helpful discussions and comments.
    The authors also thank Professor Toru Tsukioka for clarifications regarding his earlier work, and Professor Makoto Enokizono for useful discussions.

    The authors used OpenAI's ChatGPT and Anthropic's Claude during the preparation of this paper for English-language editing and discussion of mathematical ideas and arguments.
    All mathematical statements and arguments were independently verified by the authors, who take full responsibility for the content of the paper.
    
    The first author was supported by JSPS KAKENHI Grant Number JP24KJ0856 and by the WINGS-FMSP program at the Graduate School of Mathematical Sciences, the University of Tokyo.
    The second author was partially supported by JST SPRING, Grant Number JPMJSP2128.
\end{Ack}
\section{Preliminaries}
\label{section: preliminaries}

\subsection{Conormal sheaves}
In this subsection, we briefly review some basic facts about conormal sheaves of subvarieties, which will be needed in \cref{section: construction}.

Let $X$ be a variety, and let $Z\subseteq X$ be a (closed) subvariety with ideal sheaf $\IC\subseteq\OC_X$.
The \emph{conormal sheaf} $\con ZX$ of $Z$ in $X$ is the sheaf on $Z$ defined by
\[
    \con ZX = \IC/\IC^2.
\]

Let
\[
\begin{tikzcd}
    Z' \arrow[r, hookrightarrow] \arrow[d,"g"']
        & X'\arrow[d, "f"]\\
    Z \arrow[r, hookrightarrow]
        & X
\end{tikzcd}
\]
be a commutative diagram of subvarieties.
Then there is a canonical morphism
\begin{equation}
    g^*\con ZX\longrightarrow \con{Z'}{X'}.
    \label{equation: hom from pullback}
\end{equation}
If $f$ is flat and $Z'$ is the pullback $f^{-1}(Z)$ of $Z$, then \cref{equation: hom from pullback} is an isomorphism.

Let $X$ be a smooth variety, and $Z\subseteq Y\subseteq X$ two smooth subvarieties.
Then there is a short exact sequence
\[
    0\longrightarrow \con YX\rvert_Z \longrightarrow \con ZX \longrightarrow \con ZY \longrightarrow 0.
\]

We now compute the conormal sheaf in several fundamental settings.

\begin{lem}\label{lemma: conormal sheaf of a linear subspace}
    Let $X\subseteq\PP n$ be a linear subspace of codimension $c$.
    Then
    \[
        \con{X}{\PP n}\cong \OC_{X}(-1)^{\oplus c}.
    \]
\end{lem}

\begin{proof}
    We proceed by induction on $c$.
    The cases $c=0,1$ are immediate.

    Assume $c\ge2$ and let $Y\subseteq\PP n$ be a linear subspace of codimension $c-1$ containing $X$.
    Consider the exact sequence
    \begin{equation}
        \label{equation: exact sequence of conormal sheaves of two linear subspaces}
        0\longrightarrow \con Y{\PP n}\rvert_X \longrightarrow\con X{\PP n}\longrightarrow \con XY\longrightarrow 0.
    \end{equation}
    By the induction hypothesis, $\con Y{\PP n}\cong \OC_Y(-1)^{\oplus(c-1)}$ and $\con XY\cong\OC_X(-1)$.
    Since $\Ext^1(\con XY, \con Y{\PP n}\rvert_X)\cong H^1(X,\OC_X)^{\oplus(c-1)}=0$, the sequence \cref{equation: exact sequence of conormal sheaves of two linear subspaces} splits.
    Thus,
    \[
        \con{X}{\PP n}\cong \OC_{X}(-1)^{\oplus c}. \qedhere
    \]
\end{proof}


\begin{lem}\label{lemma: conormal sheaf of a fiber in a blowup}
    Let $X'$ be a smooth variety, $A'\subsetneq X'$ a smooth subvariety, $\pi\colon X\to X'$ the blowup along $A'$ with exceptional divisor $E$,
    and let $F$ be a fiber of $\pi\rvert_E\colon E\to A'$.
    Then
    \[
        \con FX\cong \OC_F(1)\oplus\OC_F^{\oplus\dim A'}.
    \]
\end{lem}

\begin{proof}
    Let $x$ be the point $\pi(F)$.
    Consider the following diagram with exact rows.
    \[
    \begin{tikzcd}
        0 \rar
            & \pi^*(\con {A'}{X'}\rvert_x) \rar \arrow[d]
            & \pi^*\con x{X'} \rar \arrow[d]
            & \pi^*\con x{A'} \rar \arrow[d]
            & 0\\
        0 \rar
            & \con EX\rvert_F \rar
            & \con FX \rar
            & \con FE \rar
            & 0.
    \end{tikzcd}
    \]
    The first row is the pullback of a sequence supported at $x$, hence split exact.
    Since $\pi\rvert_E\colon E\to A'$ is flat, the third vertical map $\pi^*\con x{A'}\to \con FE$ is an isomorphism.
    Thus, the second row also splits, and hence
    \[
        \con FX\cong\con EX\rvert_F\oplus\con FE
            \cong \OC_F(1)\oplus \OC_F^{\oplus\dim A'}. \qedhere
    \]
\end{proof}


\begin{lem}\label{lemma: restrict of conormal sheaf to a fiber in blowup}
    Let $X'$ be a smooth variety, and $A',B' \subseteq X'$ smooth subvarieties of codimension $a$, $b$, respectively.
    Assume that
    \begin{itemize}
        \item $A'\subsetneq X'$,
        \item $A'\nsupseteq B'$, and
        \item the scheme-theoretic intersection $A'\cap B'$ is irreducible, smooth, and of codimension $a+b-c$ in $X'$ with $c\ge0$.
    \end{itemize}
    Let $\pi\colon X\to X'$ be the blowup along $A'$.
    Let $B$ be the strict transform $\pi^{-1}_*B'$ of $B'$, and let $B_0$ be a fiber of $\pi\rvert_{B}\colon B\to B'$.
    Then
    \[
        \con{B}X\rvert_{B_0}\cong \OC_{B_0}^{\oplus(b-c)}\oplus \OC_{B_0}(-1)^{\oplus c}.
    \]
\end{lem}

\begin{proof}
    Let $Z'\coloneq A'\cap B'$, and $X_0$ be the fiber of $\pi\colon X\to X'$ that contains $B_0$.
    We may assume that the fiber $B_0$ lies over a point of $Z'$, as the other case is clear.
    By \cref{lemma: conormal sheaf of a fiber in a blowup}, we have
    \[
        \con {B_0}{B}\cong \OC_{B_0}(1)\oplus \OC_{B_0}^{\oplus\dim Z'}\quad\text{and}
        \quad
        \con{X_0}X\cong \OC_{X_0}(1)\oplus\OC_{X_0}^{\oplus\dim A'}.
    \]
    Since $B_0\subseteq X_0$ is a linear subspace of codimension $c$, \cref{lemma: conormal sheaf of a linear subspace} implies that 
    \[
        \con{B_0}{X_0}\cong\OC_{B_0}(-1)^{\oplus c}.
    \]
    We have exact sequences
    \begin{gather*}
        0 \longrightarrow \con{X_0}X\rvert_{B_0} \longrightarrow \con{B_0}X \longrightarrow \con {B_0}{X_0} \longrightarrow 0,\\
        0\longrightarrow \con{B}X\rvert_{B_0} \longrightarrow \con {B_0}X \longrightarrow \con {B_0}{B} \longrightarrow 0.
    \end{gather*}
    The first sequence splits since
    \[
        \Ext^1(\con{B_0}{X_0}, \con{X_0}X\rvert_{B_0})
        \cong H^1(B_0, \OC_{B_0}(2)\oplus \OC_{B_0}(1)^{\oplus\dim A'})^{\oplus c}
        =0.
    \]
    Thus, the second exact sequence becomes
    \[
        0\longrightarrow \con{B}X\rvert_{B_0} \longrightarrow \OC_{B_0}(1)\oplus\OC_{B_0}^{\oplus\dim A'}\oplus\OC_{B_0}(-1)^{\oplus c} \longrightarrow \OC_{B_0}(1)\oplus\OC_{B_0}^{\oplus\dim Z'} \longrightarrow 0.
    \]
    Here, the map sends $\OC_{B_0}(1)=\con EX\rvert_{B_0}$ isomorphically onto $\OC_{B_0}(1)=\con{E_B}B\rvert_{B_0}$, where $E$ is the exceptional divisor of $\pi$ and $E_B\coloneq E\rvert_B$, and induces $\OC_{B_0}^{\oplus\dim A'}\to\OC_{B_0}^{\oplus\dim Z'}$ as the pullback of the surjection $\con x{A'}\to\con x{Z'}$, where $x\coloneq\pi(B_0)$.
    Therefore, we conclude that
    \[
        \con{B}{X}\rvert_{B_0}
        \cong\OC_{B_0}^{\oplus(\dim A'-\dim Z')}\oplus\OC_{B_0}(-1)^{\oplus c}
        =\OC_{B_0}^{\oplus(b-c)}\oplus\OC_{B_0}(-1)^{\oplus c}. \qedhere
    \]
\end{proof}

\subsection{Nefness of divisors}

We recall a criterion for nefness of divisors due to Tsukioka
\cite{tsukioka_2023_some_examples_of_log_fano_structures_on_blowups_along_subvarieties_in_products_of_two_projective_spaces}.
This result will play a central role in \cref{section: nefness of divisors on blowups}.

\begin{lem}
    [{\cite[Lemma 3]{tsukioka_2023_some_examples_of_log_fano_structures_on_blowups_along_subvarieties_in_products_of_two_projective_spaces}}]
    \label{lemma: nefness of divisor}
    Let $X$ be a normal projective variety, and $D$ an $\RB$-Cartier divisor on $X$.
    Assume that there exists a chain of normal subvarieties
    \[
        X=X_0\supseteq X_1\supseteq\dotsb\supseteq X_m
    \]
    such that
    \begin{itemize}
        \item $X_{i+1}$ is a $\QB$-Cartier prime divisor on $X_i$ for $0\le i<m$,
        \item $D\rvert_{X_i}-X_{i+1}$ is nef for $0\le i<m$, and
        \item $D\rvert_{X_m}$ is nef.
    \end{itemize}
    Then $D$ is nef.
\end{lem}

\section{Nefness of divisors on blowups}
\label{section: nefness of divisors on blowups}

In this section, we establish several lemmas concerning the nefness of divisors on smooth blowups, and prove \cref{introthm: nefness of H-E-F}.

\begin{prop}\label{lemma: nefness of H-E}
    Let $X'$ be a smooth projective variety, $A'\subsetneq X'$ a smooth subvariety, $\pi\colon X\to X'$ the blowup along $A'$ with exceptional divisor $E$, and let $H'$ be an $\RB$-divisor on $X'$.
    Assume that there exists a chain of smooth subvarieties
    \[
        X'=X'_0\supseteq X'_1\supseteq\dotsb\supseteq X'_a=A'
    \]
    such that, for $0\le i<a$,
    \begin{itemize}
        \item $X'_{i+1}$ is a prime divisor on $X'_i$, and
        \item $H'\rvert_{X'_i}-X'_{i+1}$ is nef.
    \end{itemize}
    Then the divisor $\pi^*H'-E$ on $X$ is nef.
\end{prop}

\begin{proof}
    Define a chain of subvarieties $X=X_0\supseteq X_1\supseteq\dotsb\supseteq X_{a-1}$ by
    \[
        X_i\coloneq\pi^{-1}_*X'_i.
    \]
    By \cref{lemma: nefness of divisor}, it suffices to show that the divisor $(H-E)\rvert_{X_i}-X_{i+1}$ is nef for $0\le i<a-1$, and that $(H-E)\rvert_{X_{a-1}}$ is nef, where $H\coloneq\pi^*H'$.
    
    Let $0\le i<a-1$.
    The restriction $\pi_i\coloneq\pi\rvert_{X_i}\colon X_i\to X'_i$ is the blowup along $A'$ with exceptional divisor $E_i\coloneq E\rvert_{X_i}$, and the prime divisor $X'_{i+1}$ on $X'_i$ contains the center $A'$.
    Thus, $\pi_i^*X'_{i+1}=X_{i+1}+E_i$, and
    \[
        (H-E)\rvert_{X_i}-X_{i+1}=\pi_i^*(H'_i-X'_{i+1}),
    \]
    where $H'_i\coloneq H'\rvert_{X'_i}$.
    This shows that $(H-E)\rvert_{X_i}-X_{i+1}$ is nef for $0\le i<a-1$.
    
    Moreover, $\pi_{a-1}\colon X_{a-1}\to X'_{a-1}$ is an isomorphism identifying the divisor $E_{a-1}$ with $A'=X'_a$.
    Therefore,
    \[
        (H-E)\rvert_{X_{a-1}}
        =\pi_{a-1}^*(H'_{a-1}-X'_a),
    \]
    which implies that $(H-E)\rvert_{X_{a-1}}$ is nef.
\end{proof}

We state the following well-known fact, which follows immediately from \cref{lemma: nefness of H-E} and will be used in \cref{section: construction}.

\begin{lem}
    \label{lemma: blowup of projective space along linear subspace}
    Let $L\subsetneq \PP n$ be a linear subspace, $\pi\colon X\to\PP n$ the blowup along $L$ with exceptional divisor $E$, and let $H'$ be a hyperplane divisor on $\PP n$.
    Then $\pi^*H'-E$ is nef.
\end{lem}

\begin{proof}
    Let $\PP n=X'=X'_0\supseteq X'_1\supseteq\dotsb\supseteq X'_c=L$ be a chain of linear subspaces of codimension one.
    Since $H'\rvert_{X'_i}-X'_{i+1}\sim0$ for $0\le i<c$, \cref{lemma: nefness of H-E} implies that $\pi^*H'-E$ is nef.
\end{proof}

\begin{rem}
    \label{rem: projective bundle structure of the blowup of a projective space along a linear subspace}
    In the notation of \cref{lemma: blowup of projective space along linear subspace}, the nef divisor $\pi^*H'-E$ is base-point-free, and defines a $\PP{n-c+1}$-bundle structure on $X$ over $\PP{c-1}$, where $c$ is the codimension of $L$ in $\PP n$ (see \cite[Corollary 9.12]{book_eisenbud_harris_2016_3264_and_all_thata_second_course_in_algebraic_geometry} for instance).
\end{rem}

We now return to the general setting and consider the nefness of divisors on two successive blowups.

\begin{prop}
    \label{lemma: nefness of H-E-F}
    Let $X''$ be a smooth projective variety, and $A'',B''\subsetneq X''$ smooth subvarieties of codimension $a$, $b$ respectively.
    Assume that
    \begin{itemize}
        \item $A''\nsupseteq B''$, and
        \item $Z''\coloneq A''\cap B''$ is irreducible, smooth, and of codimension $a+b-c$ in $X''$ for some $c\ge0$.
    \end{itemize}
    Let $\pi'\colon X'\to X''$ be the blowup along $A''$ with exceptional divisor $E'$, and let $\pi\colon X\to X'$ be the blowup along $B'\coloneq (\pi')_*^{-1}B''$ with exceptional divisor $F$.
    Set $E\coloneq \pi^*E'$ and $\varphi\coloneq\pi'\circ\pi\colon X\to X''$.
    Let $H''$ be an $\RB$-divisor on $X''$.
    Assume that there exist chains of smooth subvarieties
    \begin{align*}
        X''=X''_0\supseteq X''_1\supseteq\dotsb\supseteq X''_c&=A''_c\supseteq A''_{c+1}\supseteq\dotsb\supseteq A''_a=A'',\\
        X''_c&=B''_c\supseteq B''_{c+1}\supseteq\dotsb\supseteq B''_b=B''
    \end{align*}
    such that, setting $Z''_{i,j}\coloneq A''_i\cap B''_j$ for $c\le i\le a$ and $c\le j\le b$,
    \begin{itemize}
        \item $X''_{i+1}$ is a prime divisor on $X''_i$ for $0\le i<c$,
        \item $Z''_{i+1,j}$ is a smooth prime divisor on $Z''_{i,j}$ for $c\le i<a$ and $c\le j\le b$, and $Z''_{i,j+1}$ is a smooth prime divisor on $Z''_{i,j}$ for $c\le i\le a$ and $c\le j<b$,
        \item $H''\rvert_{X''_i}-X''_{i+1}$ is nef on $X''_i$ for $0\le i<c$, and
        \item $H''\rvert_{Z''_{i,j}}-Z''_{i+1,j}-Z''_{i,j+1}$ is nef on $Z''_{i,j}$ for $c\le i<a$ and $c\le j<b$.
    \end{itemize}
    Then the divisor $\varphi^*H''-E-F$ on $X$ is nef.
\end{prop}

    The situation of \cref{lemma: nefness of H-E-F} is described by the following diagram.
    \[
    \begin{tikzcd}[column sep=-5pt, row sep=-3pt]
        X''=X''_0\supseteq X''_1\supseteq\dotsb\supseteq & X''_c & = & A''_{c} & \supseteq & A''_{c+1} & \supseteq & \dotsb & \supseteq & A''_a & = & A''\\
        & \veq && \veq && \veq &&&& \veq\\
        & B''_c & = & Z''_{c,c} & \supseteq & Z''_{c+1,c} & \supseteq & \dotsb & \supseteq & Z''_{a,c}\\
        & \vsupset && \vsupset && \vsupset &&&& \vsupset\\
        & B''_{c+1} & = & Z''_{c,c+1} & \supseteq & Z''_{c+1,c+1} & \supseteq & \dotsb & \supseteq & Z''_{a,c+1}\\
        & \vsupset && \vsupset && \vsupset &&&& \vsupset\\
        & \vdotsb && \vdotsb && \vdotsb &&&& \vdotsb\\
        & \vsupset && \vsupset && \vsupset &&&& \vsupset\\
        \phantom{X''=X''_0\supseteq X''_1\supseteq{}} B''= & B''_b & = & Z''_{c,b} & \supseteq & Z''_{c+1,b} & \supseteq & \dotsb & \supseteq & Z''_{a,b} & = & Z''
    \end{tikzcd}
    \]

\begin{proof}
    Let $X'=X'_0\supseteq X'_1\supseteq\dotsb\supseteq X'_c=B'_c\supseteq B'_{c+1}\supseteq\dotsb\supseteq B'_b=B'$ be the chain defined by
    \[
        X'_i\coloneq (\pi')^{-1}_*X''_i,\quad B'_i\coloneq(\pi')^{-1}_*B''_i.
    \]
    By \cref{lemma: nefness of H-E}, it suffices to show the following:
    \begin{enumerate}
        \item\label{item: nefness on X'_i}
            $(H'-E')\rvert_{X'_i}-X'_{i+1}$ is nef on $X'_i$ for $0\le i<c$, and
        \item\label{item: nefness on B'_i}
            $(H'-E')\rvert_{B'_i}-B'_{i+1}$ is nef on $B'_i$ for $c\le i<b$,
    \end{enumerate}
    where $H'\coloneq(\pi')^*H''$.

    We first prove \cref{item: nefness on X'_i}.
    The restriction $\pi'_i\coloneq\pi'\rvert_{X'_i}\colon X'_i\to X''_i$ is the blowup along $A''$ with exceptional divisor $E'_i\coloneq E'\rvert_{X'_i}$, and the prime divisor $X''_{i+1}$ on $X''_i$ contains the center $A''$.
    This implies that $(\pi'_i)^*X''_{i+1}=X'_{i+1}+E'_i$, and thus
    \[
        (H'-E')\rvert_{X'_i}-X'_{i+1}=(\pi'_i)^*(H''_i-X''_{i+1}),
    \]
    where $H''_i\coloneq H''\rvert_{X''_i}$.
    This is nef since $H''_i-X''_{i+1}$ is nef by assumption.

    Next we prove \cref{item: nefness on B'_i}.
    In this case, $\pi'_{B,i}\coloneq\pi'\rvert_{B'_i}\colon B'_i\to B''_i$ is the blowup along $Z''_{a,i}$ with exceptional divisor $E'_{B,i}\coloneq E'\rvert_{B'_i}$, and the prime divisor $B''_{i+1}$ on $B''_i$ does not contain the center $Z''_{a,i}$, since $Z''_{a,i}\supsetneq Z''_{a,i+1}$.
    This shows that $(\pi'_{B,i})^*B''_{i+1}=B'_{i+1}$, and thus
    \begin{equation}
        \label{equation: divisor on B'_i}
        (H'-E')\rvert_{B'_i}-B'_{i+1}=(\pi'_{B,i})^*(H''_{B,i}-B''_{i+1})-E'_{B,i},
    \end{equation}
    where $H''_{B,i}\coloneq H''\rvert_{B''_i}$.
    Consider the chain $B''_i=Z''_{c,i}\supseteq Z''_{c+1,i}\supseteq\dotsb\supseteq Z''_{a,i}$, and recall that the divisor
    \[
        (H''_{B,i}-B''_{i+1})\rvert_{Z''_{j,i}}-Z''_{j+1,i}=H''\rvert_{Z''_{j,i}}-Z''_{j+1,i}-Z''_{j,i+1}
    \]
    on $Z''_{j,i}$ is assumed to be nef for $c\le j<a$.
    By \cref{lemma: nefness of H-E} again, the divisor \cref{equation: divisor on B'_i} is nef.
\end{proof}

We conclude this section with an example illustrating how \cref{lemma: nefness of H-E,lemma: nefness of H-E-F} are applied.
While the result is due to Tsukioka \cite{tsukioka_2023_some_examples_of_log_fano_structures_on_blowups_along_subvarieties_in_products_of_two_projective_spaces}, we recover it here using the general framework established above.

\begin{eg}[cf.\,{\cite[Proposition 2]{tsukioka_2023_some_examples_of_log_fano_structures_on_blowups_along_subvarieties_in_products_of_two_projective_spaces}}]
    Let
    \begin{align*}
        X_1 & \coloneq \PP{n_1},
            & A_1 & \coloneq\{\mathrm{pt}.\}\subseteq X_1,
            & B_1 & \coloneq X_1,\\
        X_2 & \coloneq \PP{n_2},
            & A_2 & \coloneq L_d\subseteq X_2,
            & B_2 & \coloneq\{\mathrm{pt}.\}\subseteq A_2,
    \end{align*}
    where $A_2=L_d$ is a smooth hypersurface of degree $d$ in $X_2=\PP{n_2}$.

    Let $\pi'\colon X'\to X''$ be the blowup of $X''\coloneq X_1\times X_2$ along $A''\coloneq A_1\times A_2$ with exceptional divisor $E'$.
    Let $\pi\colon X\to X'$ be the blowup along $B'\coloneq(\pi')^{-1}_*B''$, where $B''\coloneq B_1\times B_2$, and let $F$ be its exceptional divisor.
    Set $E\coloneq \pi^*E'$.
    Let $\bar H_i\coloneq\OC_{\PP{n_i}}(1)$ and let $H_i$ be its pullback to $X$.
    
    Using \cref{lemma: nefness of H-E,lemma: nefness of H-E-F}, we show that $H_1+dH_2-E$ and $H_1+dH_2-E-F$ are nef.
    We first construct chains of smooth subvarieties
    \begin{align*}
        \PP{n_1} & =X_1=A_{1,0}\supseteq A_{1,1}\supseteq\dotsb\supseteq A_{1,n_1}=A_1=\{\mathrm{pt.}\},\\
        L_d & =A_2=B_{2,1}\supseteq B_{2,2}\supseteq\dotsb\supseteq B_{2,n_2}=B_2=\{\mathrm{pt.}\}
    \end{align*}
    as follows.
    Let $A_{1,j}$ be a hyperplane in $A_{1,j-1}\cong\PP{n_1-j+1}$ containing the point $A_1$ for $1\le j<n_1$.
    Let $B_{2,j}\in\bigl\lvert\bar H_2\rvert_{B_{2,j-1}} \bigr\rvert$ be a smooth irreducible element containing $B_2$ for $2\le j< n_2$.
    This yields:
    \begin{itemize}
        \item $\bar H_1\rvert_{A_{1,j}}-A_{1,j+1}\sim0$ is nef for $0\le j<n_1$,
        \item $d\bar H_2-A_2\sim0$ is nef,
        \item $d\bar H_2\rvert_{B_{2,j}}-B_{2,j+1}\sim(d-1)\bar H_2\rvert_{B_{2,j}}$ is nef for $1\le j<n_2-1$, and
        \item $d\bar H_2\rvert_{B_{2,n_2-1}}-B_2$ is a divisor of degree $d^2-1$ on the curve $B_{2,n_2-1}$, hence nef.
    \end{itemize}

    We now take the products of these chains.
    Set $X''_1\coloneq X_1\times A_2$, and
    \[
        A''_i\coloneq A_{1,i-1}\times A_2,\quad
        B''_k\coloneq X_1\times B_{2,k},\quad
        Z''_{i,k}\coloneq A''_i\cap B''_k=A_{1,i-1}\times B_{2,k}
    \]
    for $1\le i\le n_1+1$ and $1\le k\le n_2$.
    These are smooth and irreducible, and we obtain chains
    \begin{align*}
        X''\supseteq X''_1&=A''_1\supseteq A''_2\supseteq\dotsb\supseteq A''_{n_1+1}=A'',\\
        X''_1&=B''_1\supseteq B''_2\supseteq\dotsb\supseteq B''_{n_2}=B''
    \end{align*}
    as in \cref{lemma: nefness of H-E-F} with $c=1$, since $A''\cap B''$ is a point of codimension $n_1+n_2$ in $X''$.
    Let $H''_i$ be the pullback of $\bar H_i$ by the projection $X''\to X_i$.
    By the first two items above, $(H''_1+dH''_2)-X''_1\sim H''_1$ is nef, and $(H''_1+dH''_2)\rvert_{A''_i}-A''_{i+1}\sim dH''_2\rvert_{A''_i}$ is nef for $1\le i\le n_1$.
    Hence $(\pi')^*(H''_1+dH''_2)-E'$ is nef by \cref{lemma: nefness of H-E}, and so is its pullback $H_1+dH_2-E$.
    Moreover, for $1\le i\le n_1$ and $1\le k<n_2$, the divisor
    \[
        (H''_1+dH''_2)\rvert_{Z''_{i,k}}-Z''_{i+1,k}-Z''_{i,k+1}
    \]
    on $Z''_{i,k}$ is the sum of the pullbacks of $\bar H_1\rvert_{A_{1,i-1}}-A_{1,i}$ and $d\bar H_2\rvert_{B_{2,k}}-B_{2,k+1}$ by the projections, which are nef by the items above.
    Therefore $H_1+dH_2-E-F$ is nef by \cref{lemma: nefness of H-E-F}.
\end{eg}

\begin{rem}
    In fact, Tsukioka \cite[Proposition 2]{tsukioka_2023_some_examples_of_log_fano_structures_on_blowups_along_subvarieties_in_products_of_two_projective_spaces} proves the stronger statement that the nef cone of $X$ is generated by 
    \[
        H_1,\quad H_2,\quad H_1+dH_2-E,\quad H_1+dH_2-E-F.
    \]
\end{rem}

\section{General construction of varieties admitting small contractions}
\label{section: construction}

In this section, we study a sequence of blowups $X\to X'\to X''$ of a smooth variety $X''$.
Under certain geometric conditions on the centers, we show that the resulting variety $X$ admits a small contraction, thus proving \cref{introthm: existence of small contraction}.
We first fix notation and assumptions.

\begin{setup}\label{setup2 of two blowups}
    Let $X''$ be a smooth variety of dimension $\ge3$, and let $A'', B''\subseteq X''$ be smooth subvarieties of codimensions $a,b\ge2$, respectively.
    Assume that
    \begin{itemize}
        \item $A''\nsupseteq B''$,
        \item the intersection $Z''\coloneqq A''\cap B''$ is nonempty and smooth, and
        \item each irreducible component of $Z''$ has codimension $<a+b$ in $X''$.
    \end{itemize}
    Let $\pi'\colon X'\to X''$ be the blowup along $A''$ with exceptional divisor $E'$.
    Let $\pi\colon X\to X'$ be the blowup along $B'\coloneqq (\pi')^{-1}_*B''$ with exceptional divisor $F$.
    Set $E\coloneq \pi^*E'$, $W'\coloneq(\pi')^{-1}(Z'')$, $W\coloneq\pi_*^{-1}W'$, and $\varphi\coloneq\pi'\circ\pi\colon X\to X''$.
\end{setup}

The setup is summarized in the following diagram:
    \[
    \begin{tikzcd}[row sep=0pt, column sep=huge]
        X \arrow[r, "\pi=\Bl_{B'}"] &X' \arrow[r,"\pi'=\Bl_{A''}"] &X''\\
        \cup &\cup &\cup\\
        E \arrow[r,] &E' \arrow[r,"\PP{a-1}\text{-bundle}"]&A''\\
        \cup &\cup &\cup\\
        W \arrow[r, "\Bl_{W'\cap B'}"] &W' \arrow[r, "\PP{a-1}\text{-bundle}"] &Z''\\
        && \cap\\
        F \arrow[r, "\PP{b-1}\text{-bundle}"] & B' \arrow[r, "\Bl_{Z''}"] & B''   
    \end{tikzcd}
    \]

We begin by describing the structure of the fiber $X_z=\varphi^{-1}(z)$ over a point $z\in Z''$.

\begin{lem}
    \label{lemma: structure of the fiber X_z}
    In the setting of \cref{setup2 of two blowups}, let $X_z\coloneq\varphi^{-1}(z)$ be the fiber over a point $z\in Z''$.
    Suppose that the irreducible component of $Z''$ containing $z$ has codimension $a+b-c$ in $X''$, where $c\ge1$ by assumption.
    \begin{enumerate}
        \item\label{item: W_z} Let $W_z\coloneq\pi^{-1}_*W'_z$ be the strict transform of the fiber $W'_z\coloneq(\pi')^{-1}(z)$.
            Then $W_z$ is the blowup of the projective space $W'_z\cong\PP{a-1}$ along $B'_z\coloneq W'_z\cap B'$, which is a linear subspace of codimension $c$.
        \item\label{item: F_z} Let $F_z\coloneq F\cap X_z$.
            Then $F_z=\pi^{-1}(B'_z)$, which is the $\PP{b-1}$-bundle over $B'_z$:
            \[
                F_z\cong\PB_{B'_z}(\con{B'}{X'}\rvert_{B'_z}).
            \]
        \item\label{item: X_z} If $A''\subseteq B''$, then $X_z=W_z$.
            Otherwise, $X_z$ has two irreducible components: $W_z$ and $F_z$.
    \end{enumerate}
\end{lem}

\begin{proof}
    We first prove \cref{item: W_z}.
    By construction, $W'_z$ is isomorphic to $\PP{a-1}$.
    Since $\pi'\rvert_{B'}\colon B'\to B''$ is the blowup along $Z''$, considering the restriction $\pi'\rvert_{E'\cap B'}\colon E'\cap B'\to Z''$ near the point $z$, we see that the fiber $B'_z$ is a projective space $\PP{a-1-c}$, which is a linear subspace of $W'_z$.

    Next we prove \cref{item: F_z}.
    Since $\pi\rvert_F\colon F\to B'$ is the projective bundle $\PB_{B'}(\con{B'}{X'})$, its restriction $F_z$ over $B'_z$ is $\PB_{B'_z}(\con{B'}{X'}\rvert_{B'_z})$.

    Finally we prove \cref{item: X_z}.
    Since $\pi\colon X\to X'$ is the blowup along $B'$, its restriction $X_z\to W'_z$ is an isomorphism over $W'_z\setminus B'_z$.
    Thus, $W_z$ is an irreducible component of $X_z$.
    The exceptional divisor of the induced blowup $W_z\to W'_z$ is $F\rvert_{W_z}=W_z\cap F_z$, which is a $\PP{c-1}$-bundle over $B'_z$.
    On the other hand, $F_z\to B'_z$ is a $\PP{b-1}$-bundle.
    Thus, $W_z\cap F_z=F_z$ (equivalently $X_z=W_z$) if and only if $b=c$, which is equivalent to the condition $A''\subseteq B''$.
    Otherwise $F_z$ is a distinct irreducible component of $X_z$.
\end{proof}

In addition to \cref{setup2 of two blowups}, we fix the following additional notation.

\begin{setup}\label{setup: curves e and f}
    Let $f$ be a line in a fiber of the $\PP{b-1}$-bundle $\pi\rvert_F\colon F\to B'$.
    Let $e'$ be a line in a fiber of the $\PP{a-1}$-bundle $\pi'\rvert_{W'}\colon W'\to Z''$ intersecting $W'\cap B'$ but not contained in it.
    Set $e\coloneqq \pi^{-1}_*e'$.
\end{setup}

We have the following intersection numbers.

\begin{lem}\label{lemma: intersection numbers on two blowups}
    In the setting of \cref{setup2 of two blowups,setup: curves e and f}, the following hold:
    \begin{gather*}
        N^1(X/X'') =\RB[E]\oplus\RB[F],\quad
            N_1(X/X'') =\RB[e]\oplus\RB[f],\\
        E.e=-1,\quad
            E.f =0,\quad
            F.e=1,\quad
            F.f =-1.
    \end{gather*}
\end{lem}

\begin{proof}
    We prove $F.e=1$, since the other equalities are clear.
    Let $z$ be the point $\varphi(e)$.
    By construction, the curve $e$ lies in the strict transform $W_z=\pi^{-1}_* W'_z$ of the fiber $W'_z=(\pi')^{-1}(z)$.
    By \cref{lemma: structure of the fiber X_z}, $W_z$ is the blowup of the projective space $W'_z$ along the linear subspace $B'_z= W'_z\cap B'$ with exceptional divisor $F\rvert_{W_z}$.
    Since $e'=\pi_*e$ is a line in $W'_z$ intersecting $B'_z$ at a single point, we obtain
    \[
        F.e=(F\rvert_{W_z}.e)_{W_z}=1. \qedhere
    \]
\end{proof}

To prove the existence of a small contraction on $X$, we analyze the relative nef cone $\Nef(X/X'')$ and the relative cone of curves $\NEbar(X/X'')$.
A key step is verifying that the divisor $-(E+F)$ is nef over $X''$.
We begin with the following.

\begin{lem}\label{lemma: ampleness of restriction of -(E+F)}
    In the setting of \cref{setup2 of two blowups}, the restriction $-(E+F)\rvert_{X\setminus\varphi^{-1}(Z'')}$ is ample over $X''\setminus Z''$.
\end{lem}

\begin{proof}
    Let $X_x\coloneqq \varphi^{-1}(x)$ be the fiber over a point $x\in X''\setminus Z''$.
    It suffices to show that the restriction $-(E+F)\rvert_{X_x}$ is ample.
    
    If $x\notin A''\cup B''$, then $X_x$ is a point, so the assertion is trivial.

    If $x\in A''\setminus Z''$, then $X_x\cong\PP{a-1}$, $E\rvert_{X_x}\sim\OC_{\PP{a-1}}(-1)$, and $F\rvert_{X_x}=0$.
    Thus, $-(E+F)\rvert_{X_x}\sim\OC_{\PP{a-1}}(1)$, which is ample.

    Similarly, if $x\in B''\setminus Z''$, then $X_x\cong\PP{b-1}$, $E\rvert_{X_x}=0$, and $F\rvert_{X_x}\sim \OC_{\PP{b-1}}(-1)$.
    Consequently, $-(E+F)\rvert_{X_x}\sim\OC_{\PP{b-1}}(1)$, which is ample.
\end{proof}

We now describe the structure of the relative nef cone and the relative cone of curves of $X$.

\begin{prop}
    \label{proposition: structure of NE(X/Y)}
    In the setting of \cref{setup2 of two blowups,setup: curves e and f},
    \begin{align*}
        \Nef(X/X'')&=\RB^+[-E]+\RB^+[-(E+F)]\quad\text{and}\\
        \NEbar(X/X'')&=\RB^+[e]+\RB^+[f].
    \end{align*}
\end{prop}

\begin{proof}
    By \cref{lemma: intersection numbers on two blowups}, we have the intersection numbers:
    \[
        -E.f=-(E+F).e=0.
    \]
    Since $-E'$ is nef over $X''$, its pullback $-E=-\pi^*E'$ is also nef over $X''$.
    Thus, to establish the structure of the cones, it suffices to prove that the divisor $-(E+F)$ is nef over $X''$.

    In view of \cref{lemma: ampleness of restriction of -(E+F)}, it remains to check nefness over $Z''$.
    Let $X_z\coloneq\varphi^{-1}(z)$ be the fiber over a point $z\in Z''$.
    It is enough to show that the restriction $-(E+F)\rvert_{X_z}$ is nef.

    Assume that the irreducible component of $Z''$ containing $z$ has codimension $a+b-c$ in $X''$ (recall that $c\ge1$).
    By \cref{lemma: structure of the fiber X_z}, the fiber $X_z$ is the union of $W_z$ and $F_z$, where $W_z$ is the strict transform $\pi^{-1}_*W'_z$ of the fiber $W'_z=(\pi')^{-1}(z)$, and $F_z=X_z\cap F$.

    First, consider the component $W_z$.
    According to \cref{lemma: structure of the fiber X_z}, $W_z$ is the blowup of the projective space $W'_z$ along a linear subspace, with exceptional divisor $F\rvert_{W_z}$.
    Moreover,
    \[
        -E\rvert_{W_z}\sim(\pi\rvert_{W_z})^*\OC_{\PP{a-1}}(1).
    \]
    Hence \cref{lemma: blowup of projective space along linear subspace} implies that $-(E+F)\rvert_{W_z}$ is nef.

    Next, consider $F_z$.
    By \cref{lemma: structure of the fiber X_z}, $F_z$ is the $\PP{b-1}$-bundle over $B'_z=W'_z\cap B'$:
    \begin{equation}
        F_z\cong\PB_{B'_z}(\con{B'}{X'}\rvert_{B'_z}).
        \label{equation: projective bundle F_y}
    \end{equation}
    By \cref{lemma: restrict of conormal sheaf to a fiber in blowup}, the restriction of the conormal bundle is given by
    \[
        \con{B'}{X'}\rvert_{B'_z}\cong\OC_{B'_z}^{\oplus(b-c)}\oplus\OC_{B'_z}(-1)^{\oplus c}.
    \]
    Since $-F\rvert_{F_z}$ is the tautological line bundle $\OC_{F_z}(1)$ of the projective bundle \cref{equation: projective bundle F_y}, we have
    \[
        -(E+F)\rvert_{F_z}\sim
        (\pi\rvert_{F_z})^*\OC_{B'_z}(1)+\OC_{F_z}(1).
    \]
    The nefness of this divisor follows from the fact that the vector bundle
    \[
        \con{B'}{X'}\rvert_{B'_z}\otimes\mathcal O_{B'_z}(1)\cong\mathcal O_{B'_z}(1)^{\oplus(b-c)}\oplus\mathcal O_{B'_z}^{\oplus c}
    \]
    is nef.
\end{proof}

We now prove \cref{introthm: existence of small contraction}.

\begin{thm}
    \label{Theorem: existence of small contraction}
    In the setting of \cref{setup2 of two blowups,setup: curves e and f}, assume that $X''$ is quasi-projective.
    Then there exists a birational contraction $\psi\colon X\to X_0$ contracting the extremal ray $\RB^+[e]$ of $\NEbar(X/X'')$ such that $\Exc(\psi)=W$.
    Furthermore,
    \begin{enumerate}
        \item\label{item: psi small iff A not contained in B} $\psi$ is a small contraction if and only if $A''\nsubseteq B''$,
        \item\label{item: psi is K-extremal iff a>b} $\psi$ is $K_X$-extremal if and only if $a>b$, and
        \item\label{item: X_0 is blowup along union} $X_0$ is the blowup of $X''$ along $A''\cup B''$.
    \end{enumerate}
\end{thm}

\begin{proof}
    We first construct $\psi$.
    Let $\pi_0\colon X_0\to X''$ be the blowup of $X''$ along $A''\cup B''$, that is, along the ideal sheaf $\IC_{A''\cup B''}=\IC_{A''}\cap\IC_{B''}$.
    We claim the following.

    \begin{claim}\label{claim: inverse image ideal}
        The inverse image ideal sheaf $\varphi^{-1}\IC_{A''\cup B''}\cdot\OC_X$ coincides with $\OC_X(-E-F)$.
    \end{claim}

    \begin{proof}[Proof of \cref{claim: inverse image ideal}]
        Since $\pi^{-1}\IC_{B'}\cdot\OC_X=\OC_X(-F)$, it suffices to show that
        \[
            (\pi')^{-1}\IC_{A''\cup B''}\cdot\OC_{X'}=\OC_{X'}(-E')\cdot\IC_{B'}.
        \]
        This is clear over $X''\setminus Z''$, since $A''$ and $B''$ are disjoint there and $\pi'$ is an isomorphism near $B''\setminus Z''$.
        Let $z\in Z''$ be a closed point, and suppose that the irreducible component of $Z''$ containing $z$ has codimension $a+b-c$ in $X''$.
        Since $X''$, $A''$, $B''$, and $Z''$ are smooth, there exist an open neighborhood $U$ of $z$ and regular functions $x_1,\dotsc,x_c,y_{c+1},\dotsc,y_a,v_{c+1},\dotsc,v_b,w_1,\dotsc,w_r$ on $U$, forming a regular system of parameters of $\OC_{X'',z}$, such that
        \[
            \IC_{A''}\rvert_U=(x_1,\dotsc,x_c,y_{c+1},\dotsc,y_a),\quad\IC_{B''}\rvert_U=(x_1,\dotsc,x_c,v_{c+1},\dotsc,v_b).
        \]
        After shrinking $U$, we have
        \[
            \IC_{A''\cup B''}\rvert_U=(x_1,\dotsc,x_c)+(y_iv_j\mid i,j>c),
        \]
        since this equality holds in the completion of $\OC_{X'',z}$, where both sides are monomial ideals.
        The open subset $(\pi')^{-1}(U)$ is covered by the charts $U_{x_k}$ ($1\le k\le c$) and $U_{y_k}$ ($c<k\le a$), on which $E'$ is defined by $x_k$ and $y_k$, respectively.
        On $U_{x_k}$, we have $x_l=x_k\tilde x_l$ and $y_i=x_k\tilde y_i$ with $\tilde x_k=1$, and hence $(\pi')^{-1}\IC_{A''\cup B''}\cdot\OC_{U_{x_k}}=(x_k)$.
        Moreover, $B'\cap U_{x_k}=\emptyset$, since $x_k$ vanishes on $B''$ and hence $(\pi')^{-1}(B''\setminus A'')\cap U_{x_k}=\emptyset$.
        On $U_{y_k}$, we have $x_l=y_k\tilde x_l$ and $y_i=y_k\tilde y_i$ with $\tilde y_k=1$, and hence
        \[
            (\pi')^{-1}\IC_{A''\cup B''}\cdot\OC_{U_{y_k}}=y_k\cdot(\tilde x_1,\dotsc,\tilde x_c,v_{c+1},\dotsc,v_b),
        \]
        where $(\tilde x_1,\dotsc,\tilde x_c,v_{c+1},\dotsc,v_b)$ is the ideal of the strict transform $B'\cap U_{y_k}$.
        This proves \cref{claim: inverse image ideal}.
    \end{proof}

    By \cref{claim: inverse image ideal} and the universal property of blowups, there exists a morphism $\psi\colon X\to X_0$ with $\pi_0\circ\psi=\varphi$ such that $\psi^*\OC_{X_0}(1)\cong\OC_X(-E-F)$, where $\OC_{X_0}(1)\coloneq\pi_0^{-1}\IC_{A''\cup B''}\cdot\OC_{X_0}$ is invertible and $\pi_0$-ample.
    We next show the following.

    \begin{claim}\label{claim: X1 normal}
        $X_0$ is normal.
    \end{claim}

    \begin{proof}[Proof of \cref{claim: X1 normal}]
        Since $\pi_0^{-1}(X''\setminus Z'')$ is the blowup of the smooth variety $X''\setminus Z''$ along the smooth center $(A''\setminus Z'')\sqcup(B''\setminus Z'')$, it is smooth.
        Thus, it suffices to prove that, for any closed point $z\in Z''$, the localization $X_0\times_{X''}\Spec\OC_{X'',z}$ is normal.
        Let $\pf_A,\pf_B\subseteq\OC_{X'',z}$ be the prime ideals defining $A''$ and $B''$ at $z$, respectively, and let $I\coloneq\pf_A\cap\pf_B$ be the ideal defining $A''\cup B''$.
        Then
        \[
            X_0\times_{X''}\Spec\OC_{X'',z}\cong\Proj R(I),
        \]
        where $R(I)=\bigoplus_{n\ge0}I^n$ is the Rees algebra.
        Thus, it suffices to prove that $R(I)$ is normal.
        By \cite[Corollary 5.4.6]{vasconcelos_1994_arithmetic_of_blowup_algebras}, it suffices to show that $I$ is normally torsion-free, that is, $\Ass(\OC_{X'',z}/I^n)=\Ass(\OC_{X'',z}/I)$ for any $n\ge1$.
        We will show that
        \[
            I^n=\pf_A^n\cap\pf_B^n,
        \]
        which implies that
        \[
            \Ass(\OC_{X'',z}/I^n)=
            \begin{cases}
                \{\pf_B\} & \text{if }\pf_A\supseteq\pf_B\text{ (or equivalently }A''\subseteq B''),\\
                \{\pf_A,\pf_B\} & \text{otherwise.}
            \end{cases}
        \]
        Since $X''$, $A''$, $B''$, and $Z''$ are smooth, we can choose a regular system of parameters $x_1,\dotsc,x_c,y_{c+1},\dotsc,y_a,v_{c+1},\dotsc,v_b,w_1,\dotsc,w_r$ of $\OC_{X'',z}$ such that
        \[
            \pf_A=(x_1,\dotsc,x_c,y_{c+1},\dotsc,y_a),\quad\pf_B=(x_1,\dotsc,x_c,v_{c+1},\dotsc,v_b).
        \]
        Passing to the completion, we may assume that $\OC_{X'',z}$ is complete, i.e.,
        \[
            \OC_{X'',z}\cong\CB[[x_1,\dotsc,x_c,y_{c+1},\dotsc,y_a,v_{c+1},\dotsc,v_b,w_1,\dotsc,w_r]].
        \]
        Since $\pf_A$ and $\pf_B$ are monomial ideals in these coordinates, it suffices to verify the equality $I^n=~\pf_A^n\cap\pf_B^n$ for monomials.
        
        Let $f\in\OC_{X'',z}$ be a monomial in the variables $x,y,v,w$.
        Then $f$ belongs to the ideal
        \[
            I^n=\bigl((x_1,\dotsc,x_c)+(y_i v_j\mid i,j>c)\bigr)^n
        \]
        if and only if
        \[
            \deg_xf+\min\{\deg_yf,\deg_vf\}\ge n,
        \]
        which is equivalent to $f\in\pf_A^n\cap\pf_B^n$.
        This proves $I^n=\pf_A^n\cap\pf_B^n$, concluding the proof of \cref{claim: X1 normal}.
    \end{proof}

    Since $\psi$ is a projective birational morphism and $X_0$ is normal by \cref{claim: X1 normal}, we have $\psi_*\OC_X=\OC_{X_0}$ by Zariski's main theorem.
    Let $C\subseteq X$ be a curve contracted by $\varphi$.
    Since $-(E+F)$ is the pullback of a $\pi_0$-ample divisor, $C$ is contracted by $\psi$ if and only if $-(E+F).C=0$.
    By \cref{proposition: structure of NE(X/Y)}, we can write $[C]=\alpha[e]+\beta[f]$ with $\alpha,\beta\ge0$, and then $-(E+F).C=\beta$ by \cref{lemma: intersection numbers on two blowups}.
    Hence $\psi$ is the contraction of the extremal ray $\RB^+[e]$ of $\NEbar(X/X'')$, and \cref{item: X_0 is blowup along union} holds by construction.
    Moreover, the canonical divisor of $X$ is given by
    \[
        K_X=\varphi^*K_{X''}+(a-1)E+(b-1)F,
    \]
    and thus $K_X.e=b-a$ by \cref{lemma: intersection numbers on two blowups}.
    This implies \cref{item: psi is K-extremal iff a>b}.

    We now show that $\Exc(\psi)=W$.
    Recall that a curve $C\subseteq X$ contracted by $\varphi$ is contracted by $\psi$ if and only if $-(E+F).C=0$.
    In this case, the point $z\coloneq\varphi(C)$ lies in $Z''$ by \cref{lemma: ampleness of restriction of -(E+F)}.
    Moreover, for each closed point $z\in Z''$, the fiber $X_z=\varphi^{-1}(z)$ is the union of $W_z$ and $F_z$.
    Here $W_z$ is the strict transform $\pi^{-1}_*W'_z$ of the fiber $W'_z=(\pi')^{-1}(z)$, and $F_z=F\cap X_z$, see \cref{lemma: structure of the fiber X_z}.
    
    To prove $\Exc(\psi)\supseteq W$, observe that $W\cap\pi^{-1}(W'_z)=W_z$ for $z\in Z''$, and hence $W$ is covered by the $\{W_z\}_{z\in Z''}$.
    Thus, it suffices to show that $\Exc(\psi)\supseteq W_z$ for each $z\in Z''$.
    By \cref{lemma: structure of the fiber X_z}, the restriction $\pi_{W_z}\coloneq\pi\rvert_{W_z}\colon W_z\to W'_z$ is the blowup of the projective space $W'_z$ along a linear subspace with exceptional divisor $F\rvert_{W_z}$.
    Moreover,
    \[
        -(E+F)\rvert_{W_z}\sim\pi_{W_z}^*\OC_{W'_z}(1)-F\rvert_{W_z}.
    \]
    Consequently, $W_z$ is covered by curves $C$ satisfying $-(E+F).C=0$ (see \cref{rem: projective bundle structure of the blowup of a projective space along a linear subspace}), i.e., by curves contracted by $\psi$.
    Hence, we have $\Exc(\psi)\supseteq W_z$.
    
    Conversely, to prove $\Exc(\psi)\subseteq W$, let $C\subseteq X$ be a curve contracted by $\psi$, that is, $-(E+F).C=0$.
    Recall that the point $z\coloneq\varphi(C)$ lies in $Z''$, and the fiber $X_z$ containing $C$ is the union $F_z\cup W_z$.
    It suffices to show that $C\subseteq W_z$, and thus we may assume that $C\subseteq F_z$.
    Suppose that the irreducible component of $Z''$ containing $z$ has codimension $a+b-c$ in $X''$.
    Note that $c\ge1$ by assumption.
    By \cref{lemma: structure of the fiber X_z}, $F_z$ is the $\PP{b-1}$-bundle over $B'_z=W'_z\cap B'$
    \[
        F_z\cong\PB_{B'_z}(\con{B'}{X'}\rvert_{B'_z}).
    \]
    Moreover, \cref{lemma: restrict of conormal sheaf to a fiber in blowup} implies
    \[
        \con{B'}{X'}\rvert_{B'_z}\cong\OC_{B'_z}^{\oplus(b-c)}\oplus\OC_{B'_z}(-1)^{\oplus c}.
    \]
    On the other hand, since $F_z\cap W_z$ is the exceptional divisor of the blowup $W_z\to W'_z$ along $B'_z$, it is given by the projective bundle
    \[
        F_z\cap W_z\cong\PB_{B'_z}(\con{B'_z}{W'_z}).
    \]
    Since the center $B'_z$ is a linear subspace of $W'_z\cong\PP{a-1}$ of codimension $c$ by \cref{lemma: structure of the fiber X_z}, we have $\con{B'_z}{W'_z}\cong\OC_{B'_z}(-1)^{\oplus c}$ by \cref{lemma: conormal sheaf of a linear subspace}.
    The inclusion $F_z\cap W_z\hookrightarrow F_z$ of two projective bundles over $B'_z$ induces a surjection
    \[
    \begin{tikzcd}
        \con{B'}{X'}\rvert_{B'_z} \arrow[r, two heads]
            &\con{B'_z}{W'_z},
    \end{tikzcd}
    \]
    corresponding to the projection $\OC_{B'_z}^{\oplus(b-c)}\oplus\OC_{B'_z}(-1)^{\oplus c}\to\OC_{B'_z}(-1)^{\oplus c}$.
    The curve $C\subseteq F_z$ induces a surjection on $C$
    \begin{equation}
    \label{equation: surjection corresponing to curve C}
    \begin{tikzcd}
        \OC_C^{\oplus(b-c)}\oplus\pi_C^*\OC_{B'_z}(-1)^{\oplus c}\cong\pi_C^*(\con{B'}{X'}\rvert_{B'_z}) \arrow[r, two heads]
            & \OC_C(-F),
    \end{tikzcd}
    \end{equation}
    where $\pi_C\coloneq \pi\rvert_C\colon C\to B'_z$.
    Note that $\pi(C)$ is a curve: otherwise $C$ is contained in a fiber of $\pi\rvert_F$, and then $-(E+F).C=-F.C>0$.
    Hence $\pi_C^*\OC_{B'_z}(1)$ has positive degree.
    Since $-E\rvert_{F_z}\sim(\pi\rvert_{F_z})^*\OC_{B'_z}(1)$ and $-(E+F).C=0$, we have
    \begin{align*}
        \Hom_{\OC_C}(\OC_C^{\oplus(b-c)},\OC_C(-F))
        &\cong H^0(C,\OC_C(-F))^{\oplus(b-c)}\\
        &\cong H^0(C, \OC_C(-(E+F))\otimes\pi_C^*\OC_{B'_z}(-1))^{\oplus(b-c)}\\
        &=0.
    \end{align*}
    Thus, the surjection \cref{equation: surjection corresponing to curve C} factors through $\pi_C^*\OC_{B'_z}(-1)^{\oplus c}$:
    \[
    \begin{tikzcd}[column sep=small, row sep=small]
        \pi_C^*(\con{B'}{X'}\rvert_{B'_z}) \arrow[rr, two heads]\arrow[rd, two heads]
            && \OC_C(-F)\\
        & \pi_C^*(\con{B'_z}{W'_z}) \arrow[ru, two heads]
    \end{tikzcd}
    \]
    This means that the curve $C$ is contained in $F_z\cap W_z$, as required.

    Finally, we prove \cref{item: psi small iff A not contained in B}.
    Recall that $Z''$ is the disjoint union of subvarieties $Z''_i$ of codimension $a+b-c_i$ with $c_i\ge1$.
    Since $\pi'\rvert_{W'}\colon W'\to Z''$ is a $\PP{a-1}$-bundle, each $W'_i\coloneq(\pi')^{-1}(Z''_i)$ has codimension $b-c_i+1$.
    Thus, $\psi$ is a small contraction if and only if $c_i<b$.
    Note that $c_i=b$ corresponds to $Z''=A''$, that is, $A''\subseteq B''$.
    Hence $\psi$ is small if and only if $A''\nsubseteq B''$.
\end{proof}

With the notation of \cref{setup2 of two blowups}, if $A''\nsubseteq B''$ and $a>b$, the flip of the $K_X$-extremal small contraction $\psi\colon X\to X_0$ (from \cref{Theorem: existence of small contraction}) is obtained by interchanging the roles of $A''$ and $B''$.
This establishes \cref{introthm: flip}.

\begin{cor}
    \label{corollary: flip}
    In the setting of \cref{setup2 of two blowups}, assume further that $X''$ is quasi-projective and $A''\nsubseteq B''$.
    Let $X^+$ be the successive blowup of $X''$ along $B''$ and along the strict transform of $A''$.
    \begin{enumerate}
        \item If $a=b$, then the induced map $X\dashrightarrow X^+$ is a flop.
        \item If $a>b$, then $X\dashrightarrow X^+$ is the flip of the $K_X$-extremal small contraction $\psi\colon X\to X_0$ described in \cref{Theorem: existence of small contraction}.
    \end{enumerate}
\end{cor}

\begin{proof}
    Since $A''\nsubseteq B''$, the assumptions of \cref{setup2 of two blowups} remain valid after interchanging $A''$ and $B''$, and then $X^+$ plays the role of $X$.
    Let $e^+\subseteq X^+$ be the curve of \cref{setup: curves e and f} for $X^+$.
    By \cref{Theorem: existence of small contraction}, $\psi$ and the contraction $\psi^+\colon X^+\to X_0$ of $\RB^+[e^+]$ are small contractions to the same $X_0=\Bl_{A''\cup B''}X''$, and $K_{X^+}.e^+=a-b$ as in its proof.
    If $a>b$, then $K_{X^+}$ is $\psi^+$-ample, and hence $X\dashrightarrow X^+$ is the flip of $\psi$.
    If $a=b$, then $K_X$ and $K_{X^+}$ are numerically trivial over $X_0$, while $E$ is $\psi$-antiample and its strict transform $F^+$, the exceptional divisor of $X^+$ over $A''$, is $\psi^+$-ample, since $E.e=-1$ and $F^+.e^+=1$ by \cref{lemma: intersection numbers on two blowups}.
    Thus $X\dashrightarrow X^+$ is a flop.
\end{proof}

\section{Applications to products of two del Pezzo surfaces}
\label{section: weak Fano 4-fold}

In this section, we apply the blowup construction introduced in \cref{setup2 of two blowups} to products of two smooth del Pezzo surfaces.
Using the results of \cref{section: nefness of divisors on blowups,section: construction}, we prove \cref{introthm: weak Fano 4-fold of Picard number 8 with a small contraction}.

\begin{setup}
    \label{setup: product of del Pezzo}
    Let $0\le r_1,r_2\le8$ be integers.
    Let $X_i$ ($i=1,2$) be the blowup of $\PP2$ at $r_i$ points in general position, $a_1\in X_1$ be a general point, $A_2\subseteq X_2$ be the strict transform of a general line in $\PP2$, and let $b_2\in A_2$ be a general point.
    Set $X''\coloneq X_1\times X_2$, $A''\coloneq\{a_1\}\times A_2$, and $B''\coloneq X_1\times\{b_2\}$.
    Note that $X_i$ is a smooth del Pezzo surface of degree $9-r_i$, and $A''\cap B''$ is the point $z\coloneq(a_1,b_2)$.
    
    Let $\pi'\colon X'\to X''$ be the blowup along $A''$ with exceptional divisor $E'$, and let $\pi\colon X\to X'$ be the blowup along $B'\coloneq(\pi')^{-1}_*B''$ with exceptional divisor $F$.
    Set $E\coloneq\pi^*E'$ and $\varphi\coloneq\pi'\circ\pi$.
    Let $f$ be a fiber of the $\PP1$-bundle $\pi\rvert_F\colon F\to B'$.
    Let $e'$ be a line in $(\pi')^{-1}(z)\cong\PP2$ not contained in $B'$, and set $e\coloneq \pi^{-1}_*e'$.
    These are consistent with \cref{setup2 of two blowups,setup: curves e and f}.

    Let $\bar H_i$ be the pullback of the class of a line by $X_i\to\PP2$, and let $H_i$ be its pullback to $X$.
    If $r_i\ge1$, let $\bar E_{i,1},\dotsc,\bar E_{i,r_i}$ be the exceptional curves of $X_i\to\PP2$, and let $E_{i,j}$ be their pullbacks to $X$.
    Note that
    \[
        H_1,E_{1,1},\dotsc,E_{1,r_1},H_2,E_{2,1},\dotsc,E_{2,r_2},E,F
    \]
    form a basis of $N^1(X)$.

    For each $i=1,2$, define a set $S_i$ of curves in $X$ as follows:
    \begin{itemize}
        \item Fix a point $x_2\in A_2\setminus\{b_2\}$.
            Let $\mu_1\colon\tilde X_1\coloneq\Bl_{a_1}X_1\to X_1$ be the blowup of the del Pezzo surface $X_1$ at the point $a_1$ with exceptional curve $\tilde E_{1,0}$.
            Let $\tilde S_1$ be a set of curves in $\tilde X_1$ whose numerical classes generate $\NEbar(\tilde X_1)$ as a closed convex cone.
            If $1\le r_1\le7$, then $\tilde S_1$ can be chosen to be the set of ($-1$)-curves on $\tilde X_1$.
            Then define
            \[
                S_1\coloneq\set{\varphi^{-1}_*((\mu_1)_*\tilde C_1\times\{x_2\}) | \tilde C_1\in\tilde S_1\setminus\{\tilde E_{1,0}\}}.
            \]
        \item Let $\bar S_2$ be a set of curves $\bar C_2\subseteq X_2$ with $b_2\notin\bar C_2$ such that their numerical classes $[\bar C_2]$ generate $\NEbar(X_2)$ (recall that $b_2$ is a general point of the del Pezzo surface $X_2$).
            Define
            \[
                S_2\coloneq\set{\varphi^{-1}_*(\{a_1\}\times\bar C_2) | \bar C_2\in\bar S_2}.
            \]
    \end{itemize}

    Moreover, let $\bar T_1$ be a set of divisors on $X_1$ such that
    \[
        \mu_1^*\Nef(X_1)+\RB^+[\mu_1^*\bar D_1-\tilde E_{1,0} \mid \bar D_1\in\bar T_1]=\Nef(\tilde X_1).
    \]
    Define
    \[
        T\coloneq \set{D_1+H_2-E | \bar D_1\in\bar T_1}\cup\set{D_1+H_2-E-F | \bar D_1\in\bar T_1},
    \]
    where $D_1$ denotes the pullback of $\bar D_1$ to $X$.
    Note that $\bar T_1$ and $T$ can be chosen to be finite if $r_1\le7$.
\end{setup}

The following intersection computations follow directly from the setup.

\begin{lem}
    \label{lemma: intersections on blowup of product of del Pezzos}
    In the setting of \cref{setup: product of del Pezzo}, let $C_i\in S_i$ for $i=1,2$.
    Let $D$ be an $\mathbb R$-divisor on $X$, and write $D=D_1+D_2+\alpha E+\beta F$, where each $D_i$ is the pullback of an $\mathbb R$-divisor $\bar D_i$ on $X_i$, and $\alpha,\beta\in\mathbb R$.
    Then
    \[
        D.C_1=(\mu_1^*\bar D_1+\alpha\tilde E_{1,0}).\tilde C_1,\qquad D.C_2=(\bar D_2+\alpha \bar H_2).\bar C_2,
    \]
    where $\tilde C_1\in\tilde S_1$, $C_1=\varphi^{-1}_*((\mu_1)_*\tilde C_1\times\{x_2\})$, $\bar C_2\in\bar S_2$, and $C_2=\varphi^{-1}_*(\{a_1\}\times\bar C_2)$.
\end{lem}

\begin{proof}
    Let $\bar C_1\coloneq(\mu_1)_*\tilde C_1$ and $C'_i\coloneq\pi_*C_i$.
    Since $x_2\ne b_2$ and $b_2\notin\bar C_2$, the curves $\varphi_*C_1=\bar C_1\times\{x_2\}$ and $\varphi_*C_2=\{a_1\}\times\bar C_2$ do not meet $B''$, and thus $F.C_i=0$.
    By the projection formula, we obtain
    \[
        D.C_1=\bar D_1.\bar C_1+\alpha E'.C'_1,\qquad D.C_2=\bar D_2.\bar C_2+\alpha E'.C'_2.
    \]
    We compute $E'.C'_i$ on a surface containing $C'_i$.
    The strict transform of $X_1\times\{x_2\}$ in $X'$ is the blowup of $X_1\cong X_1\times\{x_2\}$ at $a_1$, which is $\tilde X_1$.
    On this surface, $C'_1$ is $\tilde C_1$ and $E'$ restricts to $\tilde E_{1,0}$, and hence $E'.C'_1=\tilde E_{1,0}.\tilde C_1$.
    The strict transform of $\{a_1\}\times X_2$ in $X'$ is isomorphic to $X_2$, since $A''$ is a divisor on this surface.
    On this surface, $C'_2$ is $\bar C_2$ and $E'$ restricts to $A_2$, and hence $E'.C'_2=A_2.\bar C_2=\bar H_2.\bar C_2$.
    Together with $\bar D_1.\bar C_1=\mu_1^*\bar D_1.\tilde C_1$, these give the assertion.
\end{proof}

We then determine the cone of curves and the nef cone of $X$ by applying the results from \cref{section: nefness of divisors on blowups}.

\begin{thm}
    \label{theorem: Nef and NE of blowup of product of del Pezzo}
    In the setting of \cref{setup: product of del Pezzo},
    \begin{align*}
        \NEbar(X)&=\overline{\RB^+[e,f,S_1,S_2]},\\
        \Nef(X)&=\Nef(X_1)+\Nef(X_2)+\RB^+[T].
    \end{align*}
\end{thm}

\begin{proof}
    Let $\CC\coloneq\Nef(X_1)+\Nef(X_2)+\RB^+[T]$.
    It suffices to show the following:
    \begin{enumerate}
        \item\label{item: dual of R+[e f S1 S2]} The dual cone of $\RB^+[e,f,S_1,S_2]$ is $\CC$.
        \item\label{item: T nef} The divisors in $T$ are nef.
    \end{enumerate}

    We first prove \cref{item: dual of R+[e f S1 S2]}.
    Let $D$ be an $\RB$-divisor on $X$, and write $D=D_1+D_2+\alpha E+\beta F$, where $D_i$ is the pullback of an $\RB$-divisor $\bar D_i$ on $X_i$.
    By construction and by \cref{lemma: intersection numbers on two blowups},
    \[
        D.e=-\alpha+\beta,\qquad D.f=-\beta.
    \]
    Therefore $D.e\ge0$ and $D.f\ge0$ are equivalent to $\alpha\le\beta\le0$.
    By \cref{lemma: intersections on blowup of product of del Pezzos}, the inequalities $D.C_1\ge0$ for all $C_1\in S_1$ are equivalent, together with $\alpha\le0$, to
    \[
        [\mu_1^*\bar D_1+\alpha\tilde E_{1,0}]\in\Nef(\tilde X_1).
    \]
    Indeed, the cone \(\NEbar(\tilde X_1)\) is generated by the classes of $\tilde E_{1,0}$ and the curves in $\tilde S_1\setminus\{\tilde E_{1,0}\}$, and
    \[
        (\mu_1^*\bar D_1+\alpha\tilde E_{1,0}).\tilde E_{1,0}=-\alpha.
    \]
    Similarly, the inequalities $D.C_2\ge0$ for all $C_2\in S_2$ are equivalent to
    \[
        [\bar D_2+\alpha\bar H_2]\in\Nef(X_2).
    \]
    Thus $[D]$ is in the dual cone of $\RB^+[e,f,S_1,S_2]$ if and only if
    \[
        \alpha\le\beta\le0,\qquad
        [\mu_1^*\bar D_1+\alpha\tilde E_{1,0}]\in\Nef(\tilde X_1),
        \qquad
        [\bar D_2+\alpha\bar H_2]\in\Nef(X_2).
    \]

    We now show that this cone is precisely $\CC$.
    First, note that any divisor $D\in T$ satisfies these conditions, since it is of one of the two forms $D_1+H_2-E$ and $D_1+H_2-E-F$, with $\bar D_1\in\bar T_1$.
    Moreover, since the pullbacks of nef divisors on $X_1$ and $X_2$ are nef, their classes also lie in the dual of $\RB^+[e,f,S_1,S_2]$.
    Thus $\CC$ is contained in the dual of $\RB^+[e,f,S_1,S_2]$.

    Conversely, assume that $[D]=[D_1+D_2+\alpha E+\beta F]$ is in the dual of $\RB^+[e,f,S_1,S_2]$.
    By the definition of $\bar T_1$ and the fact that $[\mu_1^*\bar D_1+\alpha\tilde E_{1,0}]\in\Nef(\tilde X_1)$, we may write
    \[
        \mu_1^*\bar D_1+\alpha\tilde E_{1,0}
        \equiv\mu_1^*\bar N_1
            +\sum_ic_i(\mu_1^*\bar G_i-\tilde E_{1,0})
    \]
    with $[\bar N_1]\in\Nef(X_1)$, $\bar G_i\in\bar T_1$, and finitely many nonzero coefficients $c_i\ge 0$.
    Comparing the coefficients of $\tilde E_{1,0}$, we get $\sum_ic_i=-\alpha$, and hence
    \[
        \bar D_1\equiv\bar N_1+\sum_ic_i\bar G_i.
    \]
    Since $\sum_ic_i=-\alpha\ge-\beta$, we can choose $c'_i,c''_i\ge0$ such that
    \[
        c_i=c'_i+c''_i,\qquad \sum_ic''_i=-\beta.
    \]
    Letting $\bar N_2\coloneq\bar D_2+\alpha\bar H_2$, which is nef, and pulling all divisors back to $X$, we obtain
    \begin{align*}
        D
        &\equiv\biggl(N_1+\sum_ic_iG_i\biggr)+(N_2-\alpha H_2)+\alpha E+\beta F\\
        &=N_1+N_2+\sum_ic'_i(G_i+H_2-E)+\sum_ic''_i(G_i+H_2-E-F),
    \end{align*}
    where $N_1$, $N_2$ and the $G_i$ are the pullbacks of $\bar N_1$, $\bar N_2$ and the $\bar G_i$ to $X$ respectively.
    Thus $[D]\in\CC$, and we have proved \cref{item: dual of R+[e f S1 S2]}.

    It remains to prove \cref{item: T nef}.
    Let $\bar G\in\bar T_1$, let $G$ and $G'$ be its pullbacks to $X$ and $X'$ respectively, and let $H'_2$ be the pullback of $\bar H_2$ to $X'$.
    Set $D'\coloneq G'+H'_2-E'$, so that
    \[
        G+H_2-E=\pi^*D',\qquad G+H_2-E-F=\pi^*D'-F.
    \]
    Let
    \[
        Y'\coloneq(\pi')^{-1}_*(X_1\times A_2)\subseteq X'.
    \]
    Since $A''=\{a_1\}\times A_2$ has codimension $2$ in $X_1\times A_2$, we have $Y'\cong\tilde X_1\times A_2$, and under this identification $E'\rvert_{Y'}$ is the pullback of $\tilde E_{1,0}$.
    Moreover, $B'\cong\tilde X_1\times\{b_2\}$ is a prime divisor on $Y'$, and $Y'\sim H'_2-E'$ as divisors on $X'$.
    Let $p_1$ and $p_2$ denote the projections from $\tilde X_1\times A_2$ to the two factors.
    Then
    \[
        D'-Y'\sim G',\qquad
        D'\rvert_{Y'}=p_1^*(\mu_1^*\bar G-\tilde E_{1,0})+p_2^*(\bar H_2\rvert_{A_2}).
    \]
    The divisor $G'$ is nef since the nefness of $\mu_1^*\bar G-\tilde E_{1,0}$ implies the nefness of $\bar G$.
    The divisor $\mu_1^*\bar G-\tilde E_{1,0}$ is nef by the choice of $\bar G\in\bar T_1$, and $\bar H_2\rvert_{A_2}$ has degree $1$ on $A_2\cong\PP1$.

    We first prove that $G+H_2-E=\pi^*D'$ is nef.
    By the above, both $D'-Y'$ and $D'\rvert_{Y'}$ are nef.
    Hence $D'$ is nef by \cref{lemma: nefness of divisor} applied to the chain $X'\supseteq Y'$, and so is $\pi^*D'$.

    We next prove that $G+H_2-E-F=\pi^*D'-F$ is nef.
    We apply \cref{lemma: nefness of H-E} to the blowup $\pi\colon X\to X'$ along $B'$, with respect to the chain $X'\supseteq Y'\supseteq B'$.
    We already know that $D'-Y'$ is nef.
    Under the identification $Y'\cong\tilde X_1\times A_2$, we have
    \[
        D'\rvert_{Y'}-B'
        =p_1^*(\mu_1^*\bar G-\tilde E_{1,0})+p_2^*(\bar H_2|_{A_2}-b_2),
    \]
    which is nef since $\deg(\bar H_2\rvert_{A_2}-b_2)=0$.
    Therefore $\pi^*D'-F$ is nef.
\end{proof}

Finally, \cref{introthm: weak Fano 4-fold of Picard number 8 with a small contraction} is obtained as a corollary of \cref{theorem: Nef and NE of blowup of product of del Pezzo}.

\begin{cor}\label{cor: characterization of Fano}
    In the setting of \cref{setup: product of del Pezzo}, $X$ is a smooth projective $4$-fold admitting a $K_X$-extremal small contraction.
    Moreover,
    \begin{enumerate}
        \item\label{item: characterization of Fano}
            $X$ is Fano if and only if $r_1=r_2=0$.
        \item\label{item: characterization of weak Fano}
            $X$ is weak Fano if and only if $r_1\le5$ and $r_2\le1$.
        \item\label{item: characterization of Fano type}
            $X$ is of Fano type if and only if
            \[
            \begin{cases}
                r_1\le 6,\;r_2\le7,&\text{or}\\
                r_1=7,\;r_2\le5.
            \end{cases}
            \]
    \end{enumerate}
\end{cor}

\begin{proof}
    By construction, $X$ is a smooth projective $4$-fold.
    Moreover, $X$ has a $K_X$-extremal small contraction by \cref{Theorem: existence of small contraction}.
    Note that
    \[
         -K_X=\sum_{i=1}^2\biggl(3H_i-\sum_{j=1}^{r_i}E_{i,j}\biggr)-2E-F,
    \]
    and for $C_i\in S_i$, $i=1,2$ (see \cref{setup: product of del Pezzo}),
    \[
        -K_X.C_1
        =\biggl(3\tilde H_1-\sum_{j=1}^{r_1}\tilde E_{1,j}-2\tilde E_{1,0}\biggr).\tilde C_1,\qquad
        -K_X.C_2
        =\biggl(\bar H_2-\sum_{j=1}^{r_2}\bar E_{2,j}\biggr).\bar C_2
    \]
    by \cref{lemma: intersections on blowup of product of del Pezzos}.
    Recall also that $-K_X.e=-K_X.f=1$ by \cref{lemma: intersection numbers on two blowups}.

    We first prove \cref{item: characterization of Fano}.
    If $r_1=r_2=0$, then $X_i=\PP2$ and we can take in \cref{setup: product of del Pezzo}
    \begin{align*}
        \tilde S_1&\coloneq\{\tilde E_{1,0},\tilde C_1\},\quad
            \tilde C_1\in\linsys{\tilde H_1-\tilde E_{1,0}},\\
        \bar S_2&\coloneq\{\bar C_2\},\quad
            \bar C_2\in\linsys{\bar H_2},
    \end{align*}
    as generators of $\NEbar(\tilde X_1)$ and $\NEbar(X_2)$.
    Then $-K_X.C_i=1$ for $i=1,2$, where $C_i\in S_i$ are the corresponding transforms as in \cref{setup: product of del Pezzo}.
    Since $\NEbar(X)$ is generated by $[e]$, $[f]$, $[C_1]$ and $[C_2]$ by \cref{theorem: Nef and NE of blowup of product of del Pezzo}, $-K_X$ is ample and hence $X$ is Fano in this case.
    Conversely, if $r_1\ge1$, we can take a ($-1$)-curve $\tilde C_1\in\linsys{\tilde H_1-\tilde E_{1,1}-\tilde E_{1,0}}$, and if $r_2\ge1$, we can take $\bar C_2\in\linsys{\bar H_2-\bar E_{2,1}}$.
    Then $-K_X.C_i=0$ for the transform $C_i$ as in \cref{setup: product of del Pezzo} for each $i=1,2$, and $X$ is not Fano if $(r_1,r_2)\ne(0,0)$.

    Next we prove \cref{item: characterization of weak Fano}.
    Assume $r_1\le5$ and $r_2\le1$.
    In \cref{setup: product of del Pezzo}, we take the set $\tilde S_1$ as:
    \begin{itemize}
        \item if $r_1=0$, then $\tilde S_1\coloneq\{\tilde E_{1,0},\tilde C_1\}$, $\tilde C_1\in\linsys{\tilde H_1-\tilde E_{1,0}}$,
        \item if $1\le r_1\le 3$, then let $\tilde S_1$ be the set of ($-1$)-curves $\tilde E_{1,j}$ for $0\le j\le r_1$ and $\tilde C_{1,j_1,j_2}\in\linsys{\tilde H_1-\tilde E_{1,j_1}-\tilde E_{1,j_2}}$ for $0\le j_1<j_2\le r_1$,
        \item if $4\le r_1\le5$, then let $\tilde S_1$ be the set of ($-1$)-curves $\tilde E_{1,j}$ for $0\le j\le r_1$, $\tilde C_{1,j_1,j_2}\in\linsys{\tilde H_1-\tilde E_{1,j_1}-\tilde E_{1,j_2}}$ for $0\le j_1<j_2\le r_1$, and $\tilde C_{1,j_1,\dotsc,j_5}\in\linsys{2\tilde H_1-\sum_{l=1}^5\tilde E_{1,j_l}}$ for $0\le j_1<\dotsb<j_5\le r_1$.
    \end{itemize}
    We check that $-K_X.C_1=1$ if $r_1=0$ as above, and
    \begin{align*}
        -K_X.e_{1,j}
            &=\biggl(3\tilde H_1-\sum_{l=1}^{r_1}\tilde E_{1,l}-2\tilde E_{1,0}\biggr).\tilde E_{1,j}
            =1\quad\text{for $1\le j\le r_1$\quad if $r_1\ge1$},\\
        -K_X.C_{1,j_1,j_2}
            &=\biggl(3\tilde H_1-\sum_{j=1}^{r_1}\tilde E_{1,j}-2\tilde E_{1,0}\biggr).\tilde C_{1,j_1,j_2}
            =\begin{cases}
                0 & (j_1=0),\\
                1 & (j_1\ge1)
            \end{cases}\quad
            \text{if $r_1\ge1$},\\
        -K_X.C_{1,j_1,\dotsc,j_5}
            &=\biggl(3\tilde H_1-\sum_{j=1}^{r_1}\tilde E_{1,j}-2\tilde E_{1,0}\biggr).\tilde C_{1,j_1,\dotsc,j_5}
            =\begin{cases}
                0 & (j_1=0),\\
                1 & (j_1\ge1)
            \end{cases}\quad
            \text{if $r_1\ge4$},
    \end{align*}
    where $e_{1,j}$, $C_{1,j_1,j_2}$ and $C_{1,j_1,\dotsc,j_5}$ are the transforms of the curves $\tilde E_{1,j}$, $\tilde C_{1,j_1,j_2}$ and $\tilde C_{1,j_1,\dotsc,j_5}$ respectively.
    Moreover, we take the set $\bar S_2$ as:
    \begin{itemize}
        \item if $r_2=0$, then $\bar S_2\coloneq\{\bar C_2\}$, $\bar C_2\in\linsys{\bar H_2}$,
        \item if $r_2=1$, then $\bar S_2\coloneq\{\bar E_{2,1},\bar C_2\}$, $\bar C_2\in\linsys{\bar H_2-\bar E_{2,1}}$.
    \end{itemize}
    Then $-K_X.C_2=1$ if $r_2=0$, and $-K_X.e_{2,1}=1$, $-K_X.C_2=0$ if $r_2=1$.
    Thus $-K_X$ is nef if $r_1\le5$ and $r_2\le1$ by \cref{theorem: Nef and NE of blowup of product of del Pezzo}.
    The bigness of $-K_X$ will follow from the proof of \cref{item: characterization of Fano type}.
    Thus $X$ is weak Fano in this case.

    Conversely, assume $r_1\ge6$ or $r_2\ge2$.
    If $r_1\ge6$, take a ($-1$)-curve
    \[
        \tilde C_1\in\biggl\lvert3\tilde H_1-\sum_{j=1}^6\tilde E_{1,j}-2\tilde E_{1,0}\biggr\rvert,
    \]
    and we get
    \[
        -K_X.C_1
        =\biggl(3\tilde H_1-\sum_{j=1}^{r_1}\tilde E_{1,j}-2\tilde E_{1,0}\biggr).\tilde C_1
        =-1.
    \]
    If $r_2\ge2$, take a ($-1$)-curve $\bar C_2\in\linsys{\bar H_2-\bar E_{2,1}-\bar E_{2,2}}$, and then
    \[
        -K_X.C_2=\biggl(\bar H_2-\sum_{j=1}^{r_2}\bar E_{2,j}\biggr).\bar C_2=-1.
    \]
    Hence $-K_X$ is not nef and $X$ is not weak Fano in this case.

    It remains to prove \cref{item: characterization of Fano type}.
    Let $\mu_2\colon\tilde X_2\coloneq\Bl_{b_2}X_2\to X_2$ be the blowup with exceptional divisor $\tilde E_{2,0}$, and let $\tilde A_2$ denote the strict transform of $A_2$ on $\tilde X_2$.
    We first assume that
    \[
        r_1\le6,\ r_2\le7,
        \qquad\text{or}\qquad
        r_1=7,\ r_2\le5.
    \]
    Then we claim the following:

    \begin{claim}\label{claim: auxiliary construction}
        There exist real numbers $t,u,v\ge 0$ and effective $\RB$-divisors $\tilde\Delta_1$ on $\tilde X_1$ and $\tilde G_2$ on $\tilde X_2$ such that the following hold:
        \begin{enumerate}
            \item\label{item: aux inequalities}
                $t+v<1$ and $u-v<1$.
            \item\label{item: aux X1}
                $\tilde E_{1,0}\nsubseteq\Supp\tilde\Delta_1$, $\tilde\Delta_1.\tilde E_{1,0}=u$, the pair $(\tilde X_1,\tilde\Delta_1+(1-t)\tilde E_{1,0})$ is log smooth and klt, and the divisor
                \[
                    -(K_{\tilde X_1}+\tilde\Delta_1+(1-t)\tilde E_{1,0})
                \]
                is ample on $\tilde X_1$.
            \item\label{item: aux X2}
                $\tilde E_{2,0}\cap\Supp\tilde G_2=\emptyset$, the pair $(\tilde X_2,\tilde G_2+t\tilde A_2)$ is klt, $\tilde G_2+\tilde A_2+\tilde E_{2,0}$ has simple normal crossing support, and the divisor
                \begin{align*}
                    \tilde\Gamma_2
                    &\coloneq-(K_{\tilde X_2}+\tilde G_2+t\tilde A_2)-(2-t-u)\tilde H_2+(1-t-v)\tilde E_{2,0}\\
                    &\equiv-(K_{\tilde X_2}+\tilde G_2)-(2-u)\tilde H_2+(1-v)\tilde E_{2,0}
                \end{align*}
                is ample on $\tilde X_2$.
        \end{enumerate}
    \end{claim}

    \begin{proof}
        First assume that $r_1\le 6$ and $r_2\le 7$.
        Let $0<t\ll 1$, put $u\coloneq 2(1-t)$, and take $0<v<1-t$ sufficiently close to $1-t$.
        Then \cref{item: aux inequalities} holds.

        Since $r_1\le 6$, the linear system
        \[
            \linsys{-K_{\tilde X_1}-\tilde E_{1,0}}
            =\linsys{\mu_1^*(-K_{X_1})-2\tilde E_{1,0}}
        \]
        is nonempty.
        Indeed, imposing a singular point at the general point $a_1$ gives three independent linear conditions on the linear system $\linsys{-K_{X_1}}$ of plane cubics through the $r_1$ blownup points, whose dimension is $9-r_1\ge3$.
        A general member $\tilde C_1$ of this system is the strict transform of a member of $\linsys{-K_{X_1}}$ that has an ordinary double point at $a_1$ and is smooth elsewhere.
        If $r_1=6$, this member is unique: it is the intersection of the anticanonical model $X_1\subseteq\PP3$ with the tangent plane at $a_1$, which has an ordinary double point at $a_1$ and is irreducible, since the general point $a_1$ lies neither on a line of $X_1$ nor on the parabolic curve of $X_1$.
        Hence $\tilde C_1$ is a smooth curve meeting $\tilde E_{1,0}$ transversally at the two points corresponding to the branches of the node, and $\tilde C_1+\tilde E_{1,0}$ is simple normal crossing.

        Set $\tilde\Delta_1\coloneq(1-t)\tilde C_1$.
        Then $(\tilde X_1,\tilde\Delta_1+(1-t)\tilde E_{1,0})$ is log smooth and klt, $\tilde E_{1,0}\nsubseteq\Supp\tilde\Delta_1$, and
        \[
            \tilde\Delta_1.\tilde E_{1,0}=2(1-t)=u.
        \]
        Moreover,
        \[
            -(K_{\tilde X_1}+\tilde\Delta_1+(1-t)\tilde E_{1,0})
            \equiv-tK_{\tilde X_1},
        \]
        which is ample.
        This proves \cref{item: aux X1}.

        We set $\tilde G_2\coloneq 0$.
        The curve $\tilde A_2$ is smooth and meets $\tilde E_{2,0}$ transversally at one point; hence $\tilde A_2+\tilde E_{2,0}$ is simple normal crossing, and the pair $(\tilde X_2,t\tilde A_2)$ is klt.
        Moreover,
        \[
            \tilde\Gamma_2
            \equiv-K_{\tilde X_2}-2t\tilde H_2+(1-v)\tilde E_{2,0}.
        \]
        is ample since $\tilde X_2$ is a del Pezzo surface and $t,1-v>0$ are sufficiently small.
        This proves \cref{item: aux X2}.

        Next assume that $r_1=7$ and $r_2\le 5$.
        Let $t>1/2$ be sufficiently close to $1/2$, put $u\coloneq 3(1-t)$, and take $0<v<1-t$ sufficiently close to $1-t$.
        Then
        \[
            u-1=2-3t<v<1-t,
        \]
        hence $t+v<1$ and $u-v<1$, proving \cref{item: aux inequalities}.

        Take the $(-1)$-curve
        \[
            \tilde C_1\in
            \linsys{-2K_{\tilde X_1}-\tilde E_{1,0}}
            =\biggl\lvert
                6\tilde H_1-\sum_{j=1}^7 2\tilde E_{1,j}-3\tilde E_{1,0}
            \biggr\rvert.
        \]
        Since $a_1\in X_1$ is general, $\tilde X_1$ is a del Pezzo surface of degree $1$, and hence $\tilde C_1$ is a smooth rational curve with
        \[
            \tilde C_1.\tilde E_{1,0}
            =(-2K_{\tilde X_1}-\tilde E_{1,0}).\tilde E_{1,0}
            =3.
        \]
        We show that these three points are distinct, that is, the curve $\mu_1(\tilde C_1)$, which is the unique member of $\linsys{-2K_{X_1}}$ with multiplicity $\ge3$ at $a_1$, has an ordinary triple point at $a_1$.
        This is an open condition on $a_1$, and thus it suffices to find one point $a_1$ for which it holds.
        For this, let $X_1\to\PP2$ be the double cover defined by $\linsys{-K_{X_1}}$, branched along a smooth quartic curve $R$.
        A local computation using this double cover shows that the condition holds if the image of $a_1$ is sufficiently close to a point of $R$ that is not a flex of $R$.
        Therefore, for a general point $a_1$, the curve $\tilde C_1$ meets $\tilde E_{1,0}$ transversally at three distinct points, and hence $\tilde C_1+\tilde E_{1,0}$ is simple normal crossing.
        Set $\tilde\Delta_1\coloneq(1-t)\tilde C_1$.
        Then $\tilde E_{1,0}\nsubseteq\Supp\tilde\Delta_1$ and $\tilde\Delta_1.\tilde E_{1,0}=3(1-t)=u$, and $(\tilde X_1,\tilde\Delta_1+(1-t)\tilde E_{1,0})$ is log smooth and klt.
        Moreover,
        \[
            -(K_{\tilde X_1}+\tilde\Delta_1+(1-t)\tilde E_{1,0})
            \equiv-(2t-1)K_{\tilde X_1},
        \]
        which is ample.
        This proves \cref{item: aux X1}.

        If $r_2 \le 4$, we set $\tilde G_2 \coloneq 0$.
        Then $\tilde A_2 + \tilde E_{2,0}$ is simple normal crossing, the pair $(\tilde X_2, t\tilde A_2)$ is klt as before, and
        \[
            \tilde\Gamma_2
            \equiv-K_{\tilde X_2}-(3t-1)\tilde H_2+(1-v)\tilde E_{2,0}.
        \]
        At the limiting values $t = 1/2$ and $v = 1/2$, this divisor equals
        \[
            -K_{\tilde X_2} - \frac12(\tilde H_2 - \tilde E_{2,0}).
        \]
        It has positive intersection with every $(-1)$-curve on $\tilde X_2$, since $\tilde H_2 - \tilde E_{2,0}$ has intersection $0$ or $1$ with each such curve  because $r_2 + 1 \le 5$.
        Hence the divisor is ample, and this remains so after a sufficiently small perturbation of $t$ and $v$.
        This proves \cref{item: aux X2} when $r_2 \le 4$.

        It remains to treat the case $r_2=5$.
        Take the $(-1)$-curve
        \[
            \tilde C_2\in
            \biggl\lvert2\tilde H_2-\sum_{j=1}^5\tilde E_{2,j}\biggr\rvert.
        \]
        Since $b_2$ is a general point of the general line $A_2$, we have $b_2\notin \mu_2(\tilde C_2)$ and thus $\tilde E_{2,0}\cap\tilde C_2=\emptyset$.
        Fix a sufficiently small real number $\alpha$ with $0<\alpha<1$, and set $\tilde G_2\coloneq \alpha\tilde C_2$.
        Here $\alpha$ is chosen before $t$, and $t>1/2$ is then taken sufficiently close to $1/2$; in particular, $6t-3<\alpha$.
        Since $\tilde A_2$ is the strict transform of a general line, it meets $\tilde C_2$ transversally at points away from $\tilde E_{2,0}$, and $\tilde A_2$ meets $\tilde E_{2,0}$ transversally at one point not on $\tilde C_2$.
        Therefore $\tilde C_2+\tilde A_2+\tilde E_{2,0}$ is simple normal crossing, and the pair $(\tilde X_2,\tilde G_2+t\tilde A_2)$ is klt since $\alpha,t<1$.

        Moreover, the divisor
        \[
            \tilde\Gamma_2\equiv-(K_{\tilde X_2}+\tilde G_2)-(3t-1)\tilde H_2+(1-v)\tilde E_{2,0}
        \]
        is ample.
        Indeed, at the limiting values $t=1/2$, $v=1/2$, it becomes
        \[
            -(K_{\tilde X_2}+\alpha\tilde C_2)-\frac12(\tilde H_2-\tilde E_{2,0}),
        \]
        which has positive intersection with every $(-1)$-curve on $\tilde X_2$ other than $\tilde C_2$, since the intersection number of $\tilde H_2-\tilde E_{2,0}$ with each such curve is $0$ or $1$, and $\alpha>0$ is taken sufficiently small.
        With $\tilde C_2$ itself, the intersection number is $\alpha>0$.
        Hence this divisor is ample, and it follows that $\tilde\Gamma_2$ remains ample after a sufficiently small perturbation of $t$ and $v$.
    \end{proof}

    We now return to the proof of \cref{cor: characterization of Fano}.
    Since the divisor $\tilde\Gamma_2$ is ample on $\tilde X_2$ by \cref{claim: auxiliary construction} \cref{item: aux X2}, we may choose a sufficiently small ample $\RB$-divisor $\Theta_2$ on $X_2$ such that $\tilde\Gamma_2-\mu_2^*\Theta_2$ remains ample.
    Since $\tilde G_2+\tilde A_2+\tilde E_{2,0}$ has simple normal crossing support by \cref{claim: auxiliary construction} \cref{item: aux X2}, we may take an effective $\RB$-divisor
    \[
        \tilde\Lambda_2\equiv\tilde\Gamma_2-\mu_2^*\Theta_2
    \]
    with coefficients $<1$, without common component with $\tilde G_2+\tilde A_2+\tilde E_{2,0}$, and such that
    \[
        \tilde\Lambda_2+\tilde G_2+\tilde A_2+\tilde E_{2,0}
    \]
    has simple normal crossing support.
    In particular, the pair
    \[
        (\tilde X_2,\tilde\Delta_2),\qquad\tilde\Delta_2\coloneq\tilde\Lambda_2+\tilde G_2+t\tilde A_2
    \]
    is klt.

    We define $\bar\Delta_1\coloneq(\mu_1)_*\tilde\Delta_1$.
    Since $\tilde E_{1,0}\nsubseteq\Supp\tilde\Delta_1$ and $\tilde\Delta_1.\tilde E_{1,0}=u$ by \cref{claim: auxiliary construction} \cref{item: aux X1}, we have $\mu_1^*\bar\Delta_1=\tilde\Delta_1+u\tilde E_{1,0}$, and in particular $\mult_{a_1}\bar\Delta_1=u$.

    We define $\bar\Delta_2\coloneq(\mu_2)_*\tilde\Delta_2$.
    Since $\tilde E_{2,0}\cap\Supp\tilde G_2=\emptyset$ by \cref{claim: auxiliary construction} \cref{item: aux X2} and $\tilde E_{2,0}\nsubseteq\Supp\tilde\Lambda_2$, we have $(\mu_2)^{-1}_*\bar\Delta_2=\tilde\Delta_2$ and
    \[
        \mult_{b_2}\bar\Delta_2
        =\tilde\Delta_2.\tilde E_{2,0}
        =(\tilde\Gamma_2-\mu_2^*\Theta_2+\tilde G_2+t\tilde A_2).\tilde E_{2,0}
        =t+v,
    \]
    and hence
    \begin{equation}
        \label{eq: pullback of Delta2}
        \mu_2^*\bar\Delta_2=\tilde\Delta_2+(t+v)\tilde E_{2,0}.
    \end{equation}
    Moreover,
    \[
        (1+t+u)\bar H_2-\sum_{j=1}^{r_2}\bar E_{2,j}-\bar\Delta_2
        \equiv
        \Theta_2,
    \]
    which is ample.
    Note also that $A_2$ appears in $\bar\Delta_2$ with coefficient $t$, since $\tilde\Lambda_2$ and $\tilde G_2$ do not contain $\tilde A_2$.

    Put $\Delta''\coloneq p_1^*\bar\Delta_1+p_2^*\bar\Delta_2$ and $\Delta\coloneq\varphi^{-1}_*\Delta''$.
    Along $A''=\{a_1\}\times A_2$, the divisor $p_1^*\bar\Delta_1$ has multiplicity $u=\mult_{a_1}\bar\Delta_1$, while $p_2^*\bar\Delta_2$ has multiplicity $t$.
    Moreover, since the second center $B'$ is the strict transform of $B''=X_1\times\{b_2\}$ and $\pi'$ is an isomorphism at the generic point of $B''$, we have
    \[
        \mult_{B'}(\pi')^*p_1^*\bar\Delta_1=0,\qquad
        \mult_{B'}(\pi')^*p_2^*\bar\Delta_2=\mult_{b_2}\bar\Delta_2=t+v.
    \]
    Consequently,
    \[
        \varphi^*\Delta'' = \Delta+(t+u)E+(t+v)F,
    \]
    and thus
    \begin{equation}
        \label{eq: crepant on X}
        K_X+\Delta = \varphi^*(K_{X''}+\Delta'') +(2-t-u)E+(1-t-v)F.
    \end{equation}

    We prove that $-(K_X+\Delta)$ is ample.
    By \cref{lemma: intersection numbers on two blowups},
    \[
        -(K_X+\Delta).f=1-t-v>0,\qquad
        -(K_X+\Delta).e=1-u+v>0,
    \]
    where the inequalities follow from \cref{claim: auxiliary construction} \cref{item: aux inequalities}.

    Next let $C_1\in S_1$.
    By \cref{lemma: intersections on blowup of product of del Pezzos}, we have
    \begin{align*}
        -(K_X+\Delta).C_1
        &=-(\mu_1^*(K_{X_1}+\bar\Delta_1)+(2-t-u)\tilde E_{1,0}).\tilde C_1\\
        &=-(K_{\tilde X_1}+\tilde\Delta_1+(1-t)\tilde E_{1,0}).\tilde C_1.
    \end{align*}
    The divisor $-(K_{\tilde X_1}+\tilde\Delta_1+(1-t)\tilde E_{1,0})$ is ample on $\tilde X_1$ by \cref{claim: auxiliary construction} \cref{item: aux X1}.
    Hence $-(K_X+\Delta).C_1>0$ for every $C_1\in S_1$.

    Similarly, let $C_2\in S_2$.
    By \cref{lemma: intersections on blowup of product of del Pezzos} again, we obtain
    \[
        -(K_X+\Delta).C_2
        =-(K_{X_2}+\bar\Delta_2+(2-t-u)\bar H_2).\bar C_2.
    \]
    By the construction of $\bar\Delta_2$ and $\tilde\Gamma_2$, we have
    \begin{align*}
        &-(K_{X_2}+\bar\Delta_2+(2-t-u)\bar H_2)\\
        &\equiv-(K_{X_2}+(\mu_2)_*(\tilde\Gamma_2-\mu_2^*\Theta_2+\tilde G_2+t\tilde A_2)+(2-t-u)\bar H_2)\\
        &=\Theta_2,
    \end{align*}
    which is ample.
    Therefore $-(K_X+\Delta).C_2>0$ for every $C_2\in S_2$.

    Since $\NEbar(X)$ is generated by $[e]$, $[f]$ and the classes of curves in $S_1$ and $S_2$ by \cref{theorem: Nef and NE of blowup of product of del Pezzo}, the divisor $-(K_X+\Delta)$ is ample.

    It remains to prove that $(X,\Delta)$ is klt.
    Let $\tilde X''\coloneq\tilde X_1\times\tilde X_2$ with projections $\tilde p_i\colon\tilde X''\to\tilde X_i$, and
    \[
        \tilde E''_i\coloneq\tilde p_i^*\tilde E_{i,0},\qquad
        \tilde\Delta''\coloneq\tilde p_1^*\tilde\Delta_1+\tilde p_2^*\tilde\Delta_2.
    \]
    Since $\mu_1^*\bar\Delta_1=\tilde\Delta_1+u\tilde E_{1,0}$ and by \cref{eq: pullback of Delta2}, we have
    \[
        (\mu'')^*\Delta''
        =\tilde\Delta''+u\tilde E''_1+(t+v)\tilde E''_2,
    \]
    where $\mu''\coloneq\mu_1\times\mu_2\colon\tilde X''\to X''$, and thus
    \begin{equation}
        \label{eq: crepant on tilde X}
        K_{\tilde X''}+\tilde\Delta''+(u-1)\tilde E''_1+(t+v-1)\tilde E''_2
        =(\mu'')^*(K_{X''}+\Delta'').
    \end{equation}
    The divisor $\tilde\Delta''+\tilde E''_1+\tilde E''_2$ has simple normal crossing support, since $X''$ is a product and both $\tilde\Delta_1+\tilde E_{1,0}$ and $\tilde\Delta_2+\tilde E_{2,0}$ have simple normal crossing support by \cref{claim: auxiliary construction}.

    Let $\tilde\pi'\colon\tilde X'\to\tilde X''$ be the blowup along $Z_1\coloneq\tilde p_2^*\tilde A_2\cap\tilde E''_1=\tilde E_{1,0}\times\tilde A_2$ with exceptional divisor $\tilde E'$, and let $\tilde E'_i\coloneq(\tilde\pi')^{-1}_*\tilde E''_i$.
    Let $\tilde\pi\colon\tilde X\to\tilde X'$ be the blowup along $Z_2\coloneq\tilde E'_1\cap\tilde E'_2$ with exceptional divisor $\tilde G$.
    We set $\tilde\varphi\coloneq\tilde\pi'\circ\tilde\pi\colon\tilde X\to\tilde X''$, and let $\tilde E_i$ and $\tilde E$ be the strict transforms of $\tilde E'_i$ and $\tilde E'$ on $\tilde X$, respectively.
    Since $\tilde\pi'$ and $\tilde\pi$ are permissible blowups with respect to the divisor $\tilde\Delta''+\tilde E''_1+\tilde E''_2$ with simple normal crossing support, the divisor
    \[
        \tilde\varphi^{-1}_*\tilde\Delta''+\tilde E_1+\tilde E_2+\tilde E+\tilde G
    \]
    on $\tilde X$ has simple normal crossing support.

    We now have the following diagram:
    \[
    \begin{tikzcd}[row sep=tiny,column sep=scriptsize]
        &&\tilde X''\arrow[rrdd,"\mu''"]\\
        &\tilde X'\arrow[ru,"\tilde\pi'"]\\
        \tilde X\arrow[ru,"\tilde\pi"]\arrow[rruu,bend left=35,"\tilde\varphi"]&&&& X''\\
        &&& X'\arrow[ru,"\pi'"']\\
        && X\arrow[ru,"\pi"']\arrow[rruu,bend right=35,"\varphi"']
    \end{tikzcd}
    \]
    It follows that $\IC_{A''}\OC_{\tilde X}=\OC_{\tilde X}(-\tilde E-\tilde G)$, and in particular this ideal sheaf is invertible.
    Hence the morphism $\mu''\circ\tilde\varphi\colon\tilde X\to X''$ factors through $X'$.
    Moreover, since $\IC_{B'}\OC_{\tilde X}=\OC_{\tilde X}(-\tilde E_2)$ is invertible, the resulting morphism $\tilde X\to X'$ further factors through $X$.
    Thus we get a morphism $\mu\colon\tilde X\to X$ as in the following diagram:
    \[
    \begin{tikzcd}[row sep=tiny,column sep=scriptsize]
        &&\tilde X''\arrow[rrdd,"\mu''"]\\
        &\tilde X'\arrow[ru,"\tilde\pi'"]\\
        \tilde X\arrow[ru,"\tilde\pi"]\arrow[rruu,bend left=35,"\tilde\varphi"]\arrow[rrdd,"\mu"']&&&& X''\\
        &&& X'\arrow[ru,"\pi'"']\\
        && X\arrow[ru,"\pi"']\arrow[rruu,bend right=35,"\varphi"']
    \end{tikzcd}
    \]
    Furthermore, we have $\mu^*E=\tilde E+\tilde G$ and $\mu^*F=\tilde E_2$.

    Note that
    \[
        \tilde\varphi^*\tilde\Delta''
        =\tilde\varphi^{-1}_*\tilde\Delta''+t\tilde E.
    \]
    Indeed, $\tilde p_2^*\tilde A_2$ is the only component of $\tilde\Delta''$ that contains $Z_1$, and $(\tilde\pi')^*\tilde\Delta''$ does not contain $Z_2$.
    Combining this with \cref{eq: crepant on X,eq: crepant on tilde X}, we compute
    \begin{align*}
        &\mu^*(K_X+\Delta)\\
        &=\mu^*\varphi^*(K_{X''}+\Delta'')+(2-t-u)\mu^*E+(1-t-v)\mu^*F\\
        &=\tilde\varphi^*(K_{\tilde X''}+\tilde\Delta''+(u-1)\tilde E''_1+(t+v-1)\tilde E''_2)\\
            &\quad+(2-t-u)(\tilde E+\tilde G)+(1-t-v)\tilde E_2\\
        &=(K_{\tilde X}-\tilde E-\tilde G)+(\tilde\varphi^{-1}_*\tilde\Delta''+t\tilde E)+(u-1)(\tilde E_1+\tilde E+\tilde G)+(t+v-1)(\tilde E_2+\tilde G)\\
           &\quad+(2-t-u)(\tilde E+\tilde G)+(1-t-v)\tilde E_2\\
        &=K_{\tilde X}+\tilde\varphi^{-1}_*\tilde\Delta''+(u-1)\tilde E_1+(v-1)\tilde G.
    \end{align*}
    Here recall that the variety $\tilde X$ is smooth, the coefficients satisfy $u-1<1$ and $v-1<0$ by \cref{claim: auxiliary construction} \cref{item: aux inequalities}, and $\tilde\varphi^{-1}_*\tilde\Delta''+\tilde E_1$ has simple normal crossing support.
    Therefore $(X,\Delta)$ is klt, and thus $X$ is of Fano type.

    We now prove the converse implication in \cref{item: characterization of Fano type}.
    Assume that $X$ is of Fano type.
    Then there exists an effective $\RB$-divisor $\Delta$ on $X$ such that $(X,\Delta)$ is klt and $-(K_X+\Delta)$ is ample.
    Write
    \[
        t\coloneq\coeff_Y\Delta,\qquad
        c\coloneq\coeff_E\Delta,\qquad
        d\coloneq\coeff_F\Delta,
    \]
    where $Y=\varphi^{-1}_*(X_1\times A_2)$.
    Note that $Y\equiv H_2-E-F$ as divisors on $X$.
    Since $(X,\Delta)$ is klt, we have $0\le t,c,d<1$.
    Put
    \[
        \Delta^\circ\coloneq \Delta-tY-cE-dF .
    \]

    Choose a general point $x_2\in A_2\setminus\{b_2\}$ such that
    \[
        \hat X_1\coloneq \varphi^{-1}_*(X_1\times\{x_2\})
    \]
    is not contained in $\Supp\Delta^\circ$.
    Then $\hat X_1\cong \tilde X_1=\Bl_{a_1}X_1$.
    Let $\bar\Delta_1\coloneq(\mu_1)_*(\Delta^\circ\rvert_{\hat X_1})$ be the push-forward to $X_1$ under this identification.
    We set $u\coloneq\Delta^\circ\rvert_{\hat X_1}.\tilde E_{1,0}$, the intersection number on $\hat X_1=\tilde X_1$.
    Since $N^1(\tilde X_1)=\mu_1^*N^1(X_1)\oplus\RB[\tilde E_{1,0}]$, we have
    \[
        \Delta^\circ\rvert_{\hat X_1}
        \equiv\mu_1^*\bar\Delta_1-u\tilde E_{1,0}
    \]
    on $\hat X_1=\tilde X_1$.
    Note that this holds even if $\tilde E_{1,0}$ is contained in $\Supp(\Delta^\circ\rvert_{\hat X_1})$, which cannot be excluded by the choice of $x_2$.
    By \cref{lemma: intersections on blowup of product of del Pezzos}, we have
    \[
        -(K_X+\Delta)\rvert_{\hat X_1}
        \equiv\mu_1^*(-K_{X_1}-\bar\Delta_1)-(2-t-u+c)\tilde E_{1,0},
    \]
    which is ample on $\tilde X_1$.
    Therefore, adding the effective divisor $\Delta^\circ\rvert_{\hat X_1}$, we see that the divisor
    \[
        \mu_1^*(-K_{X_1}-\bar\Delta_1)-(2-t-u+c)\tilde E_{1,0}+\Delta^\circ\rvert_{\hat X_1}
        \equiv3\tilde H_1-\sum_{j=1}^{r_1}\tilde E_{1,j}-(2-t+c)\tilde E_{1,0}
    \]
    is big on $\tilde X_1$.

    Next choose a general point $x_1\in X_1\setminus\{a_1\}$ such that
    \[
        \hat X_2\coloneq \varphi^{-1}_*(\{x_1\}\times X_2)
    \]
    is not contained in $\Supp\Delta^\circ$.
    Then $\hat X_2\cong\tilde X_2=\Bl_{b_2}X_2$.
    Let $\bar\Delta_2\coloneq(\mu_2)_*(\Delta^\circ\rvert_{\hat X_2})$ be the push-forward to $X_2$ under this identification, and set $v\coloneq \mult_{b_2}\bar\Delta_2$.
    Note that $\Delta^\circ\rvert_{\hat X_2}\equiv\mu_2^*\bar\Delta_2-v\tilde E_{2,0}$, since the curve $\hat X_2\cap F$, which is identified with $\tilde E_{2,0}$, is a fiber of $\pi\rvert_F\colon F\to B'$ over a general point and hence is not contained in $\Supp\Delta^\circ$.

    Let $D_i$ denote the pullback to $X$ of $\bar\Delta_i$.
    A numerical class on $X$ is determined by its restrictions to $\hat X_1$ and $\hat X_2$, by the basis of $N^1(X)$ in \cref{setup: product of del Pezzo} and since
    \[
        E\rvert_{\hat X_1}=\tilde E_{1,0},\quad F\rvert_{\hat X_1}=0,\quad E\rvert_{\hat X_2}=0,\quad F\rvert_{\hat X_2}=\tilde E_{2,0}.
    \]
    We therefore have $\Delta^\circ\equiv D_1+D_2-uE-vF$, and thus
    \[
        \Delta
        \equiv D_1+D_2+tH_2+(c-t-u)E+(d-t-v)F.
    \]
    Thus by \cref{lemma: intersection numbers on two blowups}, we obtain
    \[
        0<-(K_X+\Delta).f=1-t-v+d,\qquad
        0<-(K_X+\Delta).e=1-u+v+c-d.
    \]
    Therefore
    \[
        t+v-d<1,
        \qquad u-v-c+d<1.
    \]

    For every curve $C_2$ arising from a curve $\bar C_2\in\bar S_2$ as in \cref{setup: product of del Pezzo}, we have
    \[
        0<-(K_X+\Delta).C_2
        =\biggl(
            (1+u-c)\bar H_2-\sum_{j=1}^{r_2}\bar E_{2,j}-\bar\Delta_2
        \biggr).\bar C_2
    \]
    by \cref{lemma: intersections on blowup of product of del Pezzos}.
    Since the classes of curves in $\bar S_2$ generate $\NEbar(X_2)$, the divisor
    \[
        (1+u-c)\bar H_2
        -\sum_{j=1}^{r_2}\bar E_{2,j}
        -\bar\Delta_2
    \]
    is ample on $X_2$.
    Thus putting $\tilde\Delta_2\coloneq(\mu_2)^{-1}_*\bar\Delta_2=\mu_2^*\bar\Delta_2-v\tilde E_{2,0}$, the divisor
    \[
        \mu_2^*\biggl((1+u-c)\bar H_2-\sum_{j=1}^{r_2}\bar E_{2,j}-\bar\Delta_2\biggr)+\tilde\Delta_2
        =(1+u-c)\tilde H_2-\sum_{j=1}^{r_2}\tilde E_{2,j}-v\tilde E_{2,0}
    \]
    is big on $\tilde X_2$.

    In summary, we have two big divisors
    \begin{align*}
        M_1&\coloneq3\tilde H_1-\sum_{j=1}^{r_1}\tilde E_{1,j}-(2-t+c)\tilde E_{1,0},\\
        M_2&\coloneq(1+u-c)\tilde H_2-\sum_{j=1}^{r_2}\tilde E_{2,j}-v\tilde E_{2,0}
    \end{align*}
    on $\tilde X_1$ and $\tilde X_2$, respectively.

    First we exclude the case $r_1=8$.
    Suppose $r_1=8$.
    In this case, since $-K_{\tilde X_1}$ is nef, we have
    \[
        0<-K_{\tilde X_1}.M_1=1-(2-t+c),
    \]
    and thus $t>1+c$.
    This contradicts $t<1$, since $c\ge0$.
    Thus $r_1\le7$.

    Next we prove that $r_2\le7$.
    Suppose $r_2=8$.
    Then $-K_{\tilde X_2}$ is also nef, and therefore
    \[
        0<-K_{\tilde X_2}.M_2
        =3(1+u-c)-8-v
        =3(u-c)-5-v.
    \]
    On the other hand, from $u-v-c+d<1$, we have
    \[
        3(u-c)-5-v
        <3(1+v-d)-5-v
        =2v-3d-2.
    \]
    Since $t+v-d<1$, we also have
    \[
        2v-3d-2
        =2(v-d)-d-2
        <2(1-t)-d-2
        =-2t-d
        \le0.
    \]
    This is a contradiction.
    Hence $r_2\le7$.

    Assume $r_1=7$, and it now suffices to prove $r_2\le5$ in this case.
    The divisor
    \[
        N_1
        \coloneq7\tilde H_1-4\tilde E_{1,0}-3\tilde E_{1,1}-2\sum_{j=2}^{7}\tilde E_{1,j}
    \]
    on $\tilde X_1$ is nef (indeed, we can check non-negativity on ($-1$)-curves).
    Thus
    \[
        0<N_1.M_1
        =6-4(2-t+c),
    \]
    and hence $t-c>1/2$.
    Combining this with $t+v-d<1$, we obtain
    \[
        v<1-t+d<\frac12+d-c.
    \]

    Suppose now that $r_2=6$.
    The divisor
    \[
        N_2\coloneq5\tilde H_2-\tilde E_{2,0}-2\sum_{j=1}^{6}\tilde E_{2,j}
    \]
    is nef on $\tilde X_2$.
    Therefore
    \[
        0<N_2.M_2
        =5(1+u-c)-12-v,
    \]
    and hence $u-c>(7+v)/5$.
    Combining this with $u-v-c+d<1$, we have
    \[
        \frac{7+v}5<u-c<1+v-d,
    \]
    which gives
    \[
        \frac12+\frac54d<v<\frac12+d-c,
    \]
    a contradiction.
    Thus $r_2\ne6$.

    Finally, suppose that $r_2=7$.
    The divisor
    \[
        L_2\coloneq8\tilde H_2-\tilde E_{2,0}-3\sum_{j=1}^{7}\tilde E_{2,j}
    \]
    on $\tilde X_2$ is nef.
    Hence
    \[
        0<L_2.M_2
        =8(1+u-c)-21-v,
    \]
    and so $u-c>(13+v)/8$.
    Since again $u-v-c+d<1$, we get
    \[
        \frac{13+v}{8}<u-c<1+v-d,
    \]
    hence
    \[
        \frac57+\frac87d<v<\frac12+d-c.
    \]
    This is a contradiction.
    Therefore $r_2\ne7$.
    This completes the proof of \cref{item: characterization of Fano type}.
\end{proof}

\providecommand{\bysame}{}
\renewcommand{\bysame}{Toru Tsukioka}
\bibliographystyle{amsalpha}
\bibliography{bibtex_tyoshida}

@article{kawamata_1989_small_contractions_of_fourdimensional_algebraic_manifolds,
	AUTHOR = {Kawamata, Yujiro},
	TITLE = {Small contractions of four-dimensional algebraic manifolds},
	JOURNAL = {Math. Ann.},
	FJOURNAL = {Mathematische Annalen},
	VOLUME = {284},
	YEAR = {1989},
	NUMBER = {4},
	PAGES = {595--600},
	ISSN = {0025-5831,1432-1807},
	MRCLASS = {14J40 (14E30)},
	MRNUMBER = {1006374},
	DOI = {10.1007/BF01443353},
	URL = {https://doi.org/10.1007/BF01443353}
}

@article{tsukioka_2023_some_examples_of_log_fano_structures_on_blowups_along_subvarieties_in_products_of_two_projective_spaces,
	AUTHOR = {Tsukioka, Toru},
	TITLE = {Some examples of log {F}ano structures on blow-ups along
              subvarieties in products of two projective spaces},
	JOURNAL = {Geom. Dedicata},
	FJOURNAL = {Geometriae Dedicata},
	VOLUME = {217},
	YEAR = {2023},
	NUMBER = {1},
	PAGES = {Paper No. 2, 10},
	ISSN = {0046-5755,1572-9168},
	MRCLASS = {14E30 (14E15 14J45)},
	MRNUMBER = {4493663},
	MRREVIEWER = {Guolei\ Zhong},
	DOI = {10.1007/s10711-022-00735-1},
	URL = {https://doi.org/10.1007/s10711-022-00735-1}
}

@article{tsukioka_2022_on_weak_fano_manifolds_with_small_contractions_obtained_by_blowups_of_a_product_of_projective_spaces,
	AUTHOR = {Tsukioka, Toru},
	TITLE = {On weak {F}ano manifolds with small contractions obtained by
              blow-ups of a product of projective spaces},
	JOURNAL = {Adv. Geom.},
	FJOURNAL = {Advances in Geometry},
	VOLUME = {22},
	YEAR = {2022},
	NUMBER = {4},
	PAGES = {451--461},
	ISSN = {1615-715X,1615-7168},
	MRCLASS = {14J45 (14E05 14E30)},
	MRNUMBER = {4497179},
	MRREVIEWER = {Guolei\ Zhong},
	DOI = {10.1515/advgeom-2022-0007},
	URL = {https://doi.org/10.1515/advgeom-2022-0007}
}

@article{casagrande_2024_fano_4folds_with_b2_12_are_products_of_surfaces,
	AUTHOR = {Casagrande, C.},
	TITLE = {Fano 4-folds with {$b_2> 12$} are products of surfaces},
	JOURNAL = {Invent. Math.},
	FJOURNAL = {Inventiones Mathematicae},
	VOLUME = {236},
	YEAR = {2024},
	NUMBER = {1},
	PAGES = {1--16},
	ISSN = {0020-9910,1432-1297},
	MRCLASS = {14J45 (14E30 14J35)},
	MRNUMBER = {4712863},
	MRREVIEWER = {Antonio\ Lanteri},
	DOI = {10.1007/s00222-024-01236-6},
	URL = {https://doi.org/10.1007/s00222-024-01236-6}
}

@article{casagrande_codogni_fanelli_2019_the_blowup_of_bbb_p4_at_8_points_and_its_fano_model_via_vector_bundles_on_a_del_pezzo_surface,
	AUTHOR = {Casagrande, Cinzia and Codogni, Giulio and Fanelli, Andrea},
	TITLE = {The blow-up of {$\mathbb{P}^4$} at 8 points and its {F}ano
              model, via vector bundles on a del {P}ezzo surface},
	JOURNAL = {Rev. Mat. Complut.},
	FJOURNAL = {Revista Matem\'atica Complutense},
	VOLUME = {32},
	YEAR = {2019},
	NUMBER = {2},
	PAGES = {475--529},
	ISSN = {1139-1138,1988-2807},
	MRCLASS = {14J60 (14E30 14J35 14J45)},
	MRNUMBER = {3942925},
	MRREVIEWER = {Sarbeswar\ Pal},
	DOI = {10.1007/s13163-018-0282-5},
	URL = {https://doi.org/10.1007/s13163-018-0282-5}
}

@book{book_eisenbud_harris_2016_3264_and_all_thata_second_course_in_algebraic_geometry,
	AUTHOR = {Eisenbud, David and Harris, Joe},
	TITLE = {3264 and all that---a second course in algebraic geometry},
	PUBLISHER = {Cambridge University Press, Cambridge},
	YEAR = {2016},
	PAGES = {xiv+616},
	ISBN = {978-1-107-60272-4; 978-1-107-01708-5},
	MRCLASS = {14-01 (14C15 14M15 14N10)},
	MRNUMBER = {3617981},
	MRREVIEWER = {Arnaud\ Beauville},
	DOI = {10.1017/CBO9781139062046},
	URL = {https://doi.org/10.1017/CBO9781139062046}
}

@misc{casagrande_2025_towards_the_classification_of_Fano_4-folds_with_b2_7,
      title={Towards the classification of Fano 4-folds with $b_2\geq 7$}, 
      author={C. Casagrande},
      year={2025},
      eprint={2508.21207},
      archivePrefix={arXiv},
      primaryClass={math.AG},
      howpublished={https://arxiv.org/abs/2508.21207}, 
}

@book {debarre_2001_higher-dimensional_algebraic_geometry,
    AUTHOR = {Debarre, Olivier},
     TITLE = {Higher-dimensional algebraic geometry},
    SERIES = {Universitext},
 PUBLISHER = {Springer-Verlag, New York},
      YEAR = {2001},
     PAGES = {xiv+233},
      ISBN = {0-387-95227-6},
   MRCLASS = {14-02 (14E30 14Jxx)},
  MRNUMBER = {1841091},
MRREVIEWER = {Mark\ Gross},
       DOI = {10.1007/978-1-4757-5406-3},
       URL = {https://doi-org.utokyo.idm.oclc.org/10.1007/978-1-4757-5406-3},
}

@book {vasconcelos_1994_arithmetic_of_blowup_algebras,
    AUTHOR = {Vasconcelos, Wolmer V.},
     TITLE = {Arithmetic of blowup algebras},
    SERIES = {London Mathematical Society Lecture Note Series},
    VOLUME = {195},
 PUBLISHER = {Cambridge University Press, Cambridge},
      YEAR = {1994},
     PAGES = {viii+329},
      ISBN = {0-521-45484-0},
   MRCLASS = {13A30 (13-02)},
  MRNUMBER = {1275840},
MRREVIEWER = {Manfred\ Herrmann},
       DOI = {10.1017/CBO9780511574726},
       URL = {https://doi-org.utokyo.idm.oclc.org/10.1017/CBO9780511574726},
}
\end{document}